\documentclass[11pt]{article}
\usepackage{amsmath,amssymb,amsthm,amsfonts,mathtools}
\usepackage{geometry}
\usepackage{xcolor}
\usepackage[hidelinks]{hyperref}
\usepackage{cleveref}
\newtheorem{proposition}{Proposition}

\newtheorem{theorem}{Theorem}
\newtheorem*{theorem*}{Theorem}

\newtheorem{lemma}{Lemma}

\newtheorem{corollary}{Corollary}
\theoremstyle{definition}

\newtheorem{remark}{Remark}

\newcommand{\tr}{\operatorname{Tr}}

\newcommand{\Id}{\mathrm{Id}}

\newcommand{\im}{\mathrm{Im}}

\newcommand{\ic}{\mathrm{i}}

\newcommand{\arctanh}{\mathrm{arctanh}}

\definecolor{darkpurple}{rgb}{0.4, 0, 0.23}

\begin{document}

\title{Geometric functionals of Brownian motion on Hermitian symmetric spaces of non-compact type}

\author{Fabrice Baudoin\footnote{Research partially supported by grant 10.46540/4283-00175B from Independent Research Fund Denmark and by the Villum Investigator grant \emph{Stochastic Analysis in Aarhus}. The two authors acknowledge funding from the European Research Council (ERC) under the European Union's Horizon Europe research and innovation programme (RanGe project, Grant Agreement No. 101199772).}, Alexandre Reber}
\date{}
\maketitle
\begin{abstract}
We study Brownian motion on  Hermitian symmetric spaces of non-compact type in
their bounded-domain realization. Using Jordan triple systems, we identify the spectral
values after an appropriate change of variables as a Heckman--Opdam diffusion of type \(BC_r\). We then analyze two Brownian functionals: the symplectic area
associated with the genus-normalized invariant K\"ahler form, and, in the tube-type case, the winding defined by the Jordan determinant. For the area process we prove a martingale representation, a central limit theorem, and an exact conditional characteristic function expressed as a ratio of
Heckman-Opdam heat kernels. For the determinant winding process we obtain analogous
heat kernel formulas and prove convergence to a Cauchy law with scale determined by the initial
determinant.  These results extend classical formulas of Paul L\'{e}vy
and Marc Yor from the Euclidean setting to the full class of Hermitian
symmetric spaces of non-compact type.

\end{abstract}

\tableofcontents

\section{Introduction}

The study of random functionals that encode geometric and topological information
about the underlying space is a central theme in modern probability theory, with deep connections to
harmonic analysis, representation theory, mathematical physics and complex geometry. The present paper is devoted to
two such functionals defined for Brownian motion on Hermitian symmetric spaces of non-compact type: the
\emph{stochastic area process}, which measures the signed symplectic area swept by the Brownian path
relative to the K\"{a}hler structure, and the \emph{Jordan determinant winding process}, which
counts the windings of the Brownian path around a canonical complex hypersurface. Our main
results are explicit characteristic function formulas for both processes, expressed in terms of
heat kernels of Heckman--Opdam hypergeometric Laplacians, together with precise asymptotic
descriptions.

\subsection*{Background and motivation}

The first geometric functional of Brownian motion to have been studied in the literature is the stochastic area in the plane. For a Brownian motion
$(Z(t))_{t \geq 0} = (X(t), Y(t))_{t \geq 0}$ in $\mathbb{R}^2$, started at the origin, the
\emph{L\'{e}vy stochastic area} is the process
\[
  S(t) = \frac{1}{2} \int_0^t \bigl(X(s)\, dY(s) - Y(s)\, dX(s)\bigr),
\]
which represents the algebraic area enclosed between the Brownian path and its chord. In his seminal 1950 paper \cite{Levy1950}, Paul L\'{e}vy computed the characteristic function
\[
  \mathbb{E}\bigl[e^{i\alpha S(t)} \mid Z(t)=z\bigr] = \frac{\alpha t / 2}{\sinh(\alpha t/2)} \exp \left( -\frac{|z|^2}{2t} \left( \frac{\alpha t}{2} \coth \frac{\alpha t}{2} -1\right)\right)
\]
using a Fourier series expansion of $Z(t)$. In 1980, Marc Yor provided in \cite{Yor1980} a second, insightful proof of this formula. His argument combines, in an elegant way, Girsanov's theorem with explicit properties of the harmonic oscillator semigroup and its connection to Ornstein--Uhlenbeck processes. Extensions to Brownian motion on $\mathbb{C}^n$ are straightforward: for a standard complex Brownian motion on $\mathbb{C}^n$, the area formula factorises as a product of $n$ independent copies of the one-dimensional formula, reflecting the flat geometry of $\mathbb{C}^n$. Somewhat surprisingly, the stochastic area process and L\'{e}vy's formula appear in a wide range of contexts, including index theory \cite{Bismut1984}, rough paths theory \cite{FrizVictoir}, analytic number theory \cite{BianePitmanYor01}, and mathematical physics \cite{Duplantier}.

Winding numbers provide a second source of motivation. Spitzer's celebrated 1958 theorem \cite{Spitzer1958} asserts that, for planar Brownian motion started away from the origin, $2\theta(t)/\log t$
converges in distribution to a standard Cauchy random variable, where
\[
\theta(t)=\int_0^t \frac{X(s)\, dY(s) - Y(s)\, dX(s)}{X(s)^2+Y(s)^2}
\]
is the total winding angle around the origin. Using a similar approach to that of \cite{Yor1980}, Marc Yor computed in his 1980 paper \cite{Yor1980winding} the exact characteristic function of $\theta(t)$ , which reads
\[
\mathbb{E}\bigl[e^{i\alpha \theta(t)} \mid |Z(t)|=r \bigr]=\frac{\mathrm{I}_{|\alpha|}(\rho r/t)}{\mathrm{I}_0(\rho r/t)},
\]
where $r>0$, $\rho=|Z(0)|$, and $\mathrm{I}_{|\alpha|}$ denotes the modified Bessel function of the first kind.

\vspace{0.4cm}

Various extensions of Spitzer's asymptotic theorem, as well as of L\'{e}vy's area and Yor's winding formulas, can be found in the literature; see \cite{BertoinWerner1994,Franchi}, the monograph \cite{BaudoinBook2024}, and \cite{baudoin2024}. However, despite extensive work, a general theory is still lacking, and  available extensions have been developed within specific families of examples, notably those connected to random matrix theory. The present paper substantially extends this program to the full class of \emph{Hermitian symmetric spaces of non-compact type}, which simultaneously generalise the complex unit ball, the Siegel domain, matrix unit balls, and the two exceptional Hermitian symmetric spaces of non-compact type. The main novelty of the paper is threefold. First, we give a direct Jordan-theoretic derivation of the dynamics of the spectral values of the Brownian motion on bounded symmetric domains, avoiding the use of the global \(G/K\) harmonic analysis machinery. Second, we introduce a stochastic symplectic area process and show that  it admits an exact conditional characteristic function governed by a simple shift of Heckman--Opdam heat kernel multiplicities. Third, in the tube-type case, we identify a determinant winding functional that also admits an  exact conditional characteristic function but whose asymptotic distribution is Cauchy rather than Gaussian.

\subsection*{The setting}

Let $X = G/K$ be an irreducible Hermitian symmetric space of non--compact type and rank $r$. Via the Harish--Chandra embedding, $X$ can be realised as a \emph{bounded symmetric domain} \cite{Helgason1978,KoranyiWolf1965}, that is, a bounded open subset $\Omega$ of a complex vector space $V$ of complex dimension $n$, endowed with the structure of a Hermitian positive Jordan triple system $(V,\{\cdot,\cdot,\cdot\})$. We equip $\Omega$ with the genus-normalized Bergman metric. Its associated K\"{a}hler form is
\[
\omega = \frac{i}{2\gamma}\partial\bar{\partial} \log K(z,z),
\]
where $K(z,w) = \mathrm{det} B(z,w)^{-1}$ is the Bergman kernel expressed in terms of the Bergman operator $B(z,w)$. The structural parameters of $\Omega$ are its rank $r$, its characteristic multiplicities $a,b \geq 0$, and its genus $\gamma = (r-1)a + b + 2$; these are related to the complex dimension by
\begin{equation*}
  n = r + \tfrac{r(r-1)}{2}a + rb,
\end{equation*}
see \eqref{dimension}.

The Brownian motion $(X(t))_{t \geq 0}$ on $\Omega$, with respect to the
genus-normalized Bergman metric, is the unique strong solution of the It\^{o} stochastic
differential equation
\[
  dX(t) = \sqrt{\gamma}\,B(X(t), X(t))^{1/2} \, d\beta(t), \quad X(0) = z \in \Omega,
\]
see Theorem~\ref{Theorem EDS}, where $(\beta(t))_{t \geq 0}$ is a standard
complex Brownian motion on $(V,(\cdot|\cdot))$. A central result of this
paper, proved in Theorem~\ref{spectral diffusion}, is that if
\[
  \lambda_1(t) \geq \cdots \geq \lambda_r(t) \geq 0
\]
denote the spectral values of $X(t)$, that is, the singular values in the
Jordan--algebraic sense, then the transformed variables
\[
  \tau_i(t)=\operatorname{arctanh}(\lambda_i(t))
\]
form a diffusion with generator $\frac{1}{2}\mathcal{L}_{a,1,2b}$, the
radial Heckman--Opdam operator of type $BC_r$ with multiplicity parameters $(k_1,k_2,k_3)=(a,1,2b)$
see \eqref{radial generator}. From the Lie-theoretic point of view, the identification of the radial part
of Brownian motion on a Riemannian symmetric space with a
Heckman--Opdam-type diffusion is classical; see, for instance,
\cite{MalliavinMalliavin}, \cite{Schapira_sharp_estimates}, and
\cite{Taylor_1991}. However, this result does not by itself give
a direct description of the process obtained by taking the spectral values
of $X(t)$ in a Jordan frame in the bounded-domain realization. In particular,
the passage from the Lie-theoretic radial component to the Jordan spectral values
requires an explicit identification through the Harish--Chandra isomorphism. Theorem~\ref{spectral diffusion} instead provides an intrinsic derivation in the language of Hermitian Jordan triple systems. Its proof uses only the
Peirce decomposition, the Bergman operator, and the polar integration
formula, making it particularly well adapted to the subsequent analysis of
the stochastic area and determinant winding functionals.

\subsection*{Main results}

\paragraph{The stochastic area process.}
Let $\Phi(z) = \frac{1}{2\gamma}\log \mathrm{det} B(z,z)^{-1}=-\frac12\log N(z,z)$ be the canonical $K$-invariant K\"{a}hler potential on
$\Omega$, and let $\alpha = d^c\Phi = \frac{1}{2i}(\partial - \bar{\partial})\Phi$ be its
associated $d^c$-potential form. The \emph{stochastic symplectic area process} is the Stratonovich line
integral
\[
  S(t) = \int_{X[0,t]} \alpha,
\]
which is well-defined as a local martingale by the general theory of stochastic line integrals of
$d^c$-forms on K\"{a}hler manifolds developed in Section~2 (see Theorem~\ref{martingale property}). Our first main result, Corollary~\ref{martingale representation for S}, is an
explicit decomposition: there exists an $\mathbb{R}^r$-valued Brownian motion
$(\alpha_i(t))_{1 \leq i \leq r}$ such that
\[
  S(t) = \frac12 \sum_{i=1}^r \int_0^t \lambda_i(s)\, d\alpha_i(s).
\]
It  yields a central limit theorem proved in Theorem~\ref{CLT area}:

\begin{theorem*}
Let $(X(t))_{t \ge 0}$ be a Brownian motion on $\Omega$ and let $(S(t))_{t \ge 0}$ be the associated stochastic area process.
Then, as $t \to +\infty$,
\[
  \frac{1}{\sqrt{t}}\, S(t)
  \;\Longrightarrow\;
  \mathcal{N}(0,r/4),
\]
where $\mathcal{N}(0,r/4)$ denotes the centered Gaussian distribution with variance
$r/4$.
\end{theorem*}

This is obtained by combining the law of large numbers for the Heckman--Opdam process
\cite[Proposition~4.2]{Schapira_Heckman}, which forces $\lambda_i(t) \to 1$ almost
surely, with the martingale central limit theorem. Our most precise result for the stochastic area is the following
characteristic function formula obtained in Theorem~\ref{characteristic function area}:

\begin{theorem*}
For every $t > 0$, $\alpha > 0$, and $\tau \in \mathcal{C}$, choose the
continuous radial-bridge version of the conditional expectation. Then
\[
  \mathbb{E}\bigl[e^{i\alpha S(t)} \mid \tau(t) = \tau\bigr]
  = e^{n\alpha t/2}
    \frac{p^{a,\, 1+\alpha,\, 2b - \alpha}_t(\tau(0), \tau)}
         {p^{a,\, 1,\, 2b}_t(\tau(0), \tau)}
    \prod_{i=1}^r \bigl(4\cosh\tau_i(0)\cosh\tau_i\bigr)^{\alpha/2},
\]
where $p^{k_1,k_2,k_3}_t$ denotes the heat kernel of the Heckman--Opdam Laplacian
$\frac{1}{2}\mathcal{L}_{k_1,k_2,k_3}$. 
\end{theorem*}

The proof proceeds by a Girsanov change of measure, constructing an explicit exponential
martingale $D_\alpha(t)$ built from $\sum_i \ln\cosh\tau_i(t)$. A key computation using
Lemma~\ref{computation coth} shows that the resulting process after the Girsanov change of measure has generator
$\frac{1}{2}\mathcal{L}_{a,1+\alpha, 2b-\alpha}$, so that the ratio
of heat kernels appears naturally. 

\paragraph{The Jordan determinant winding process in tube domains.}
For domains of tube type (equivalently, $b = 0$ in the Peirce decomposition), the Jordan triple
system $(V, \{\cdot,\cdot,\cdot\})$ carries a canonical Jordan determinant $\mathrm{det}_{ \mathcal J}$,
whose zero set $\mathcal{Z} = \{\mathrm{det}_{ \mathcal J} = 0\}$ is a complex hypersurface in $V$. Writing
$\mathcal J(z) = |\mathrm{det}_{ \mathcal J}(z)| = \prod_{i=1}^r \lambda_i(z)$ for the modulus of the Jordan
determinant, we define the \emph{$d^c$-log-determinant form}
\[
  \theta = 2d^c \log \mathcal J = d\arg(\mathrm{det}_{ \mathcal J}),
\]
and the associated \emph{Jordan determinant winding process}
\[
  A(t) = \int_{X[0,t]} \theta.
\]
The class $[\theta/(2\pi)]$ generates $H^1(\Omega\setminus\mathcal Z;\mathbb Z)\simeq\mathbb Z$, or equivalently its image generates the integral lattice in de Rham cohomology. 
The process $A(t)$ measures, up to normalisation, the winding number of the process
$t \mapsto \mathrm{det}_{ \mathcal J}(X(t))$ around the origin in $\mathbb{C}$, and is well-defined for
$t\ge0$ whenever $X(0)\in\Omega\setminus\mathcal Z$; the polar-set property then prevents a later hit of $\mathcal Z$. The winding functional is intrinsically tied to the existence of the Jordan determinant and is therefore available only in the tube-type case.  Indeed, one can check that the form $d^c (\log\prod_{i=1}^r \lambda_i )$ is only closed in the tube-type case.

By Lemma~\ref{martingale representation for A}, there exists an explicit decomposition
\[
  A(t) = \sum_{i=1}^r \int_0^t
    \frac{1 - \lambda_i(s)^2}{\lambda_i(s)}\, d\alpha_i(s),
\]
whose quadratic variation $\langle A \rangle(t) = 4\sum_i \int_0^t
(\coth(2\tau_i(s))^2 - 1)\, ds$ has a density singular near $\mathcal Z$ and decaying as the radial process escapes. The joint process $(\tau(t), A(t))_{t \geq 0}$ is
a diffusion with an explicitly computed generator (Proposition~\ref{skew product log determinant}). The characteristic function of $A(t)$ is given by
Theorem~\ref{characteristic function dc log det} which reads:

\begin{theorem*}
For every \(t>0\), \(\alpha>0\), and \(\tau\in\mathcal C\), choose the
continuous radial-bridge version of the conditional expectation. Then
\begin{align*}
\mathbb E\left(e^{i\alpha A(t)}\mid \tau(t)=\tau\right)
& =e^{
\alpha
r\left(
2+2\alpha+a(r-1)
\right)t
}
\prod_{i=1}^r
\left(
\sinh(2\tau_i(0))\sinh(2\tau_i)
\right)^{\alpha}
\frac{
p_t^{a,\,1+2\alpha,\,0}(\tau(0),\tau)
}{
p_t^{a,\,1,\,0}(\tau(0),\tau)
} \\
&=
2^{-2\alpha r}
\prod_{i=1}^r
\left(
\tanh(\tau_i(0))\tanh(\tau_i)
\right)^{\alpha}
\frac{
p_t^{a,\,1-2\alpha,\,4\alpha}(\tau(0),\tau)
}{
p_t^{a,\,1,\,0}(\tau(0),\tau)
}.
\end{align*}
\end{theorem*}

The long-time asymptotics of the winding
process  obtained in Corollary~\ref{Cauchy limit} contrast sharply with those of the area:

\begin{theorem*}
As $t \to \infty$, the Jordan determinant winding process converges in distribution:
\[
  A(t) \;\Longrightarrow\;
  \mathcal{C}_{\log\frac{1}{\mathcal J(X(0))}},
\]
where $\mathcal{C}_s$ denotes the Cauchy distribution with scale parameter $s$.
\end{theorem*}

The scale parameter $\log(1/\mathcal J(X(0)))$ is the logarithmic determinant depth
of the starting point relative to the hypersurface $\mathcal{Z}$, consistent with the
heuristic that Brownian paths starting close to the hypersurface they wind around accumulate
more winding before escaping to infinity.

\subsection*{Comparison with existing literature}

As already noted, the study of Brownian functionals whose exact distributions can explicitly be computed, or represented in terms of hypergeometric functions, has a long history dating back to the rigorous introduction of Brownian motion itself. Even for linear Brownian motion, only a few such classes of functionals are known. These include area-type functionals and their relatives, quadratic functionals; winding functionals; and exponential functionals together with  Pitman's \(2M-X\) type functionals. We refer to \cite{YorSomeaspects,YorExponential} for a comprehensive overview.  

Exponential, Pitman's type functionals, and their remarkable identities,  admit natural analogues in geometric or Lie theoretic settings; see, for instance, \cite{BaudoinOConnell,Biane, ChhaibiMR3418540, Directedpolymers} and the more recent work \cite{Chhaibi2024arXiv241206701C}. Our results belong to this tradition and identify analogues of the stochastic area and winding functionals in the full generality of Hermitian symmetric spaces of non-compact type.

Stochastic area functionals have been considered in geometric frameworks before, see \cite{BaudoinBook2024}. In particular, among previous works, the closest antecedent to our analysis of the stochastic area functional on symmetric spaces of non-compact type is \cite{BaudoinDemniWang2023Stiefel}. For type~I, its hyperbolic Stiefel connection area agrees, up to orientation, with $2S$ under the present strict K\"ahler normalization.  
Compared with \cite{BaudoinDemniWang2023Stiefel}, which treats stochastic area on the type-I spaces \(U(p,q)/(U(p)U(q))\) by random-matrix diffusion methods, the present approach applies uniformly to all irreducible Hermitian symmetric spaces of non-compact type, including exceptional domains. Winding functionals have been considered on Riemann surfaces or on circle bundles, see \cite{baudoin2024}. For the Jordan determinant winding functional in tube domains, the only prior result concerns the rank-one type~I case  with \( n=1 \), corresponding to \(X \simeq \mathbb{C}H^1\), which was analysed in \cite[Section~1.2.4]{BaudoinBook2024} and \cite{baudoin2024}.

\subsection*{Organisation of the paper}

Section~2 develops the background on stochastic calculus for $d^c$-forms on K\"{a}hler
manifolds and establishes the key local martingale criterion (Lemma~\ref{divergence free},
Theorem~\ref{martingale property}). Section~3 reviews the Jordan-algebraic structure of bounded
symmetric domains and proves Theorem~\ref{spectral diffusion} on the spectral part of the
Brownian motion, a central structural result of the paper. Section~4 defines the stochastic
area process, establishes its martingale decomposition (Corollary~\ref{martingale representation for S}), and proves the central limit theorem (Theorem~\ref{CLT area}) and the characteristic
function formula (Theorem~\ref{characteristic function area}). Section~5 treats the Jordan
determinant winding process in tube domains: the decomposition
(Lemma~\ref{martingale representation for A}), the characteristic function
(Theorem~\ref{characteristic function dc log det}), and the Cauchy limit theorem
(Corollary~\ref{Cauchy limit}). 
Appendix A collects new  facts on Heckman--Opdam 
processes and heat kernels, including integral representations and exponential-integrability 
estimates. Those facts generalise some results by Schapira \cite{Schapira_Heckman,Schapira_sharp_estimates} to situations where the multiplicity might be negative. Appendix B gives explicit formulas for the area and winding forms in all irreducible bounded symmetric domains, 
and Appendix C contains auxiliary computations used in the proofs.

\subsection*{Frequently used notations}

\begin{center}
\begingroup
\setlength{\tabcolsep}{5pt}
\begin{tabular}{c|l}
Notation & Meaning \\
\hline
 \\
\(d^c\) & \( \frac{1}{2 \mathrm{i}} ( \partial -\bar \partial)\) \\
\( \Omega \subset V \)& bounded symmetric domain in the complex vector space $V$ \\
\( B(z,z) \) & Bergman operator \\
\( ( \cdot | \cdot) \) & Hermitian trace form in $V$\\
\(r\) & rank of the bounded symmetric domain $\Omega$ \\
\(a,b\) & characteristic multiplicities \\
\(\gamma\) & genus, \(\gamma=(r-1)a+b+2\) \\
\(\Phi\) & Genus-normalized K\"ahler potential, \(\Phi=(2\gamma)^{-1}\log K\)\\
\(\omega\) & Corresponding K\"ahler form, \(\omega=i\partial\bar\partial\Phi\) \\

\(\alpha\) & K\"ahler \(d^c\)-potential form, \(\alpha=d^c\Phi\) \\

\(S(t)\) & Stochastic symplectic area process,
\(S(t)=\int_{X[0,t]}\alpha\) \\

\(\det_{\mathcal J}\) & Jordan determinant in the tube-type case \\

\(\mathcal J(z)\) & Modulus of the Jordan determinant,
\(\mathcal J(z)=|\det_{\mathcal J}(z)|\) \\

\(\theta\) & Jordan determinant winding form,
\(\theta=2d^c\log {\mathcal J}=d\arg(\det_{\mathcal J})\) \\

\(A(t)\) & Jordan determinant winding process,
\(A(t)=\int_{X[0,t]}\theta\) \\

\(\lambda_i\) & Jordan spectral values \\
\(\tau_i\) & radial variables, \(\tau_i=\operatorname{arctanh}\lambda_i\) \\ 
\( \mathcal C \) & \( \left\{ \tau \in \mathbb{R}^r : 0< \tau_r < \cdots <  \tau_1  \right\} \) \\
\( \bar{\mathcal C} \) & \( \left\{ \tau \in \mathbb{R}^r : 0\le \tau_r \le \cdots \le  \tau_1  \right\} \) \\
\(\mathcal{L}_{k_1,k_2,k_3}\) & \(BC_r\) Heckman--Opdam Laplacian \\
 & \small{ \( \sum_{i=1}^r \frac{\partial^2}{\partial \tau_i^2}
 + \sum_{i=1}^r \Bigg( k_3 \coth(\tau_i) + 2k_2 \coth(2\tau_i)+k_1 \sum_{j\neq i} \bigl( \coth(\tau_i - \tau_j) + \coth(\tau_i + \tau_j) \bigr) \Bigg) \frac{\partial}{\partial \tau_i} \) }\\
 \(m_{k_1,k_2,k_3}\) & corresponding reference measure  \\
  & \small{\( \prod_{i=1}^r |\sinh \tau_i|^{k_3}\,|\sinh (2\tau_i)|^{k_2}\;
 \prod_{1\le i<j\le r}\Bigl|\sinh(\tau_i-\tau_j)\,\sinh(\tau_i+\tau_j)\Bigr|^{k_1} d\tau \)}  \\
\(p_t^{k_1,k_2,k_3}\) & corresponding heat kernel (see Theorem \ref{heat kernel formula})  \\
\( \int Y \circ d X \) & Stratonovich stochastic integral \\
\(  \int Y dX \) & It\^o stochastic integral \\

\(\langle M\rangle\) & Quadratic variation of a semimartingale \(M\) \\
 \\
\hline
\end{tabular}
\endgroup
\end{center}

The stochastic processes we consider are defined on a filtered probability space whose filtration $(\mathcal F_t)_{t \ge 0}$ satisfies the usual conditions.

\section{Preliminaries on stochastic calculus in  K\"ahler manifolds}

The area and winding functionals are both Stratonovich line integrals of the form $\int_{X[0,t]}d^c\Phi$ for an appropriate $\Phi$ where $X$ is a Brownian motion on a K\"ahler manifold. This section develops the necessary stochastic calculus in K\"ahler manifolds and proves the general martingale criterion for such integrals that will be applied twice later.

\subsection{K\"ahler manifolds}
\label{Preliminaries on complex geometry}

We recall the conventions from K\"ahler geometry used throughout the paper;
see \cite{moroianu_lectures} for background. Let \(M\) be a complex manifold
of complex dimension \(n\), and let \(z_j=x_j+\ic y_j\) be local holomorphic
coordinates. We use the notations
\[
\frac{\partial}{\partial z_j}
=
\frac12\left(
\frac{\partial}{\partial x_j}
-\ic\frac{\partial}{\partial y_j}
\right),
\qquad
\frac{\partial}{\partial \bar z_j}
=
\frac12\left(
\frac{\partial}{\partial x_j}
+\ic\frac{\partial}{\partial y_j}
\right),
\]
and
\[
dz_j=dx_j+\ic dy_j,
\qquad
d\bar z_j=dx_j-\ic dy_j.
\]
The complex structure \(J\) is the $(1,1)$-tensor characterized by
\[
J\frac{\partial}{\partial z_j}
=
\ic\frac{\partial}{\partial z_j},
\qquad
J\frac{\partial}{\partial \bar z_j}
=
-\ic\frac{\partial}{\partial \bar z_j}.
\]

Let
\[
h=\sum_{i,j=1}^n h_{i\bar j}\,dz_i\otimes d\bar z_j
\]
be a Hermitian metric. The associated Riemannian metric and real two-form are
\[
g=\operatorname{Re}h
=
\frac12\sum_{i,j=1}^n h_{i\bar j}
\left(dz_i\otimes d\bar z_j+d\bar z_j\otimes dz_i\right),
\]
and
\[
\omega=-\operatorname{Im}h
=
\frac{\ic}{2}
\sum_{i,j=1}^n h_{i\bar j}\,dz_i\wedge d\bar z_j.
\]
We say that \(h\) is K\"ahler if \(d\omega=0\). Equivalently, the complex
structure is parallel for the Levi-Civita connection of \(g\), that is
\(\nabla J=0\). We use the convention
\[
\omega(u,v)=g(Ju,v),
\qquad u,v\in T_xM.
\]

If \((h^{i\bar j})\) denotes the inverse matrix of \((h_{i\bar j})\), then
the Laplace--Beltrami operator of \(g\) is
\begin{align}
\label{fomula-laplacian}
\Delta
=
4\sum_{i,j=1}^n
h^{i\bar j}
\frac{\partial^2}{\partial z_i\,\partial\bar z_j}.
\end{align}
The absence of first-order terms in this expression is one of the main
characteristics of the K\"ahler setting. We write the exterior derivative \(d=\partial+\bar\partial\), with
\[
\partial^2=\bar\partial^2=0,
\qquad
\partial\bar\partial+\bar\partial\partial=0,
\]
and we define the $d^c$ operator:
\[
d^c
=
\frac{1}{2\ic}(\partial-\bar\partial).
\]
Thus, for a real-valued smooth function \(f\),
\begin{align}\label{dc formula}
d^cf
=
-\frac12\,df\circ J,
\qquad
(d^cf)^\sharp
=
\frac12 J\nabla f,
\end{align}
where \(^{\sharp}\) denotes the musical isomorphism induced by \(g\). A smooth function \(\Phi\) is a K\"ahler potential if
\[
\omega=\ic\,\partial\bar\partial\Phi.
\]
With the above definition of \(d^c\), this is equivalent to
\[
d(d^c\Phi)=\omega.
\]
Accordingly, \(d^c\Phi\) will be called the K\"ahler \(d^c\)-potential form
associated with \(\Phi\). As usual, K\"ahler potentials are local in general,
and two local potentials differ by the real part of a holomorphic function.

\subsection{Stochastic line integrals of $d^c$-forms}

Let $(X(t))_{t \ge 0}$ be a Brownian motion on a connected K\"ahler manifold $(M,\omega,g)$, that is, a diffusion with generator $\frac{1}{2} \Delta$ where $\Delta$ is the Laplace-Beltrami operator of the Riemannian metric $g$. We assume that the lifetime of $(X(t))_{t \ge 0}$ is infinite, i.e. that $M$ is stochastically complete.

\begin{lemma}\label{divergence free}
Let $\eta$ be a smooth one-form on $M$. The Stratonovich line integral $\left( \int_{X[0,t]} \eta \right)_{t\ge 0}$ is a local martingale for Brownian motion started from every point of $M$ if and only if $\eta$ is co-closed, i.e. divergence free.
\end{lemma}

\begin{proof}
    Let $(X(t))_{t \ge 0}$ be a Brownian motion on $M$ obtained from the Eells-Elworthy-Malliavin construction on the orthonormal frame bundle, see \cite[Proposition 3.2.2]{Hsu2002}.  Thanks to \cite[Definition 2.4.1]{Hsu2002}, the Stratonovich line integral $\int_{X[0,t]} \eta $ is given by
    \begin{align}\label{line integral def}
    \int_{X[0,t]} \eta =\sum_{i=1}^{2n}\int_0^t \eta (U(s)\mathrm{e}_i) \circ dW^{i}(s)
    \end{align}
    where $U(t)$ is the horizontal Brownian motion on the orthonormal frame bundle of $M$, $\mathrm{e}_i$ is the canonical basis of $\mathbb{R}^{2n}$ and $W(t)$ a Brownian motion in $\mathbb{R}^{2n}$. The bounded variation part of \eqref{line integral def} can then be computed in a local coordinate system.
    For a real local coordinate system $(x^1,\ldots,x^{2n})$, let $g^{ij},\sqrt{g}^{ij}$ denote the inverses of the metric tensor and of its square root respectively. We write 
    \[
        \eta = \sum_{i=1}^{2n}\eta_i\,dx^i
        .
    \]
    For $(B^{i}(t),i=1,\ldots,2n)_{t\ge0}$ a $2n$-dimensional Brownian motion, and $\Gamma_{ij}^{k}$ the Christoffel symbols associated to the Riemannian metric $g$, one can locally describe the dynamics of $X$ as
    \[
        dX^i(t)=\sum_{j=1}^{2n}\sqrt{g}^{ij}(X(t))dB^j(t) - \frac{1}{2}\sum_{j,k,l=1}^{2n}\sqrt{g}^{jk}\sqrt{g}^{jl}\Gamma_{kl}^{i}(X(t))dt
    \]
    up to the random time $\tau$ when $(X(t))_{t\ge0}$ leaves the open set associated with the local coordinate system.
    Relying on Example~2.4.4 in \cite{Hsu2002}, for the semimartingale $X(t)$ considered up to the exit time $\tau$, we have
    \[
        S(t)
        =
        \int_{X[0,t]} \eta
        =
        \sum_{i=1}^{2n}\int_{0}^{t}\eta_i(X(s))\circ dX^{i}(s)
        .
    \]
    Applying It\^o's formula to $\eta_i(X_s)$ and since $d\left\langle X^i,X^j \right\rangle(t) = g^{ij}(X(t))dt$, we get
    \[
        \left\langle \eta_{i}(X_{\cdot}), X_{\cdot}^i \right\rangle (t)
        =
        \sum_{j=1}^{2n}\int_{0}^{t} \frac{\partial \eta_i}{\partial x^j}(X(s))g^{ij}(X(s))ds
    \]
    Combining the three above identities, we finally obtain the existence of a local martingale $M(t)$ defined up to $\tau$ such that
    \[
        S(t) - M(t)
        =
        \frac{1}{2}\int_{0}^{t}
        \sum_{i,j=1}^{2n}g^{ij}(X(s))
        \left(  
            \frac{\partial\eta_j}{\partial x^i}- \sum_{p=1}^{2n} \eta_p \Gamma_{ij}^{p}
        \right)
        \left(X(s)\right)
        ds
        = - \frac{1}{2}\int_{0}^{t} \delta\eta\left( X(s)\right)ds
        .
    \]
   From \eqref{line integral def}, the local martingale $M(t)$ is given by the It\^o integral
   \[
   M(t)=\sum_{i=1}^{2n}\int_0^t \eta (U(s)\mathrm{e}_i) dW^{i}(s),
   \]
   which is independent of the chart we are considering.
   Using a covering of $M$ by coordinate charts we have for all $t \ge 0$
   \[
   S(t)=\sum_{i=1}^{2n}\int_0^t \eta (U(s)\mathrm{e}_i) dW^{i}(s)- \frac{1}{2}\int_{0}^{t} \delta\eta\left( X(s)\right)ds.
   \]
   If $\delta\eta=0$, this decomposition proves the local martingale property. Conversely, if the line integral is a local martingale for Brownian motion started at an arbitrary $x\in M$, uniqueness of the continuous semimartingale decomposition implies that the finite-variation term vanishes. Differentiating it at $t=0$ gives $\delta\eta(x)=0$. Since $x$ is arbitrary, $\delta\eta=0$ on $M$.
\end{proof}

\begin{theorem}\label{martingale property}
Let $\Phi:M \to \mathbb{R}$ be a smooth function. The process
\[ 
S(t):=\int_{X[0,t]} d^c \Phi
\]
is a local martingale with quadratic variation $\frac{1}{4} \int_0^t \| \nabla \Phi \|^2 (X(s))  ds$, where $\| \nabla \Phi \|$ denotes the norm of the gradient of $\Phi$  for the Riemannian metric $g$.
\end{theorem}
\begin{proof}
Let $\alpha=d^c \Phi$. Thanks to \eqref{dc formula} we have
\[
\alpha^\sharp=\frac{1}{2} J\nabla\Phi.
\]
 The local martingale property will then follow from Lemma \ref{divergence free} thanks to the fact 
that $J\nabla\Phi$ is a divergence free vector field which implies that 
$\alpha$ is co-closed. To  compute the  divergence of $J\nabla\Phi$, choose a local $g$--orthonormal real frame $\{ E_a, 1 \le a \le 2n \}$ adapted to $J$, i.e.\ $\{ E_a, 1 \le a \le 2n \}=\{e_1,\dots,e_n,Je_1,\dots,Je_n\}$.
Using $\nabla J=0$, the skew-symmetry of $J$ with respect to $g$, $J^2=-\mathrm{Id}$  and the definition $\operatorname{div}X=\sum_a g(\nabla_{E_a}X,E_a)$ for the orthonormal frame $\{E_a\}$, we get
\begin{align*}
\operatorname{div}(J\nabla\Phi)
&=\sum_{k=1}^n \Big( g(\nabla_{e_k}(J\nabla\Phi),e_k)+g(\nabla_{Je_k}(J\nabla\Phi),Je_k)\Big)\\
&=\sum_{k=1}^n \Big( g(J\nabla_{e_k}\nabla\Phi,e_k)+g(J\nabla_{Je_k}\nabla\Phi,Je_k)\Big)\\
&=\sum_{k=1}^n \Big( -g(\nabla_{e_k}\nabla\Phi,Je_k)\;-\;g(\nabla_{Je_k}\nabla\Phi,J(Je_k))\Big)\\
&=\sum_{k=1}^n \Big( -\mathrm{Hess}(\Phi)(e_k,Je_k)\;+\;\mathrm{Hess}(\Phi)(Je_k,e_k)\Big).
\end{align*}
Here $\mathrm{Hess}(\Phi)(U,V):=g(\nabla_U\nabla\Phi,V)$ is the Riemannian Hessian, which is symmetric because $\nabla$ is torsion free:
\[
\mathrm{Hess}(\Phi)(Je_k,e_k)=\mathrm{Hess}(\Phi)(e_k,Je_k).
\]
Thus each summand cancels and we conclude
\[
\operatorname{div}(J\nabla\Phi)=0,
\]
i.e.\ $\operatorname{div}(\alpha^\sharp)=0$. Equivalently, the codifferential satisfies $\delta\alpha=0$. Therefore, from Lemma \ref{divergence free}, $S(t)$ is a local martingale. Its quadratic variation is given by
\[
\left \langle S \right\rangle (t)=\int_0^t \| \alpha \|^2 (X(s)) ds=\frac{1}{4}\int_0^t \| J\nabla\Phi \|^2 (X(s)) ds=\frac{1}{4}\int_0^t \| \nabla\Phi \|^2 (X(s)) ds,
\]
where we used the fact that $J$ is an isometry.
\end{proof}

\section{Brownian motion in bounded symmetric domains}

In this section we study the Brownian motion and its spectral values in Hermitian symmetric spaces of non--compact type which are realized as bounded symmetric domains.  The main result of the section is Theorem \ref{spectral diffusion}.

\subsection{Preliminaries on bounded symmetric domains}
\label{Preliminaries on bounded symmetric domains}
We begin by recalling some standard material on bounded symmetric domains and Jordan algebras; for further background see \cite{Faraut-Koranyi,FarautKoranyi1994, Book_complexdomains,Loos1977}. Let $X = G/K$ be an irreducible Hermitian symmetric space of non--compact type and rank $r$. Via the Harish--Chandra embedding, $X$ is realized as a bounded symmetric domain $\Omega \subset V$, where $V$ is a complex vector space of complex dimension $n$ endowed with the structure of a \emph{(positive) Hermitian Jordan triple system}.

Recall that a Hermitian Jordan triple system is a complex vector space $V$ equipped with a triple product
\[
\{x\,y\,z\} \in V,
\]
which is $\mathbb{C}$--bilinear in $x$ and $z$, conjugate--linear in $y$, symmetric in $x$ and $z$, and satisfies the Jordan triple identity
\[
\{x\,y\,\{u\,v\,w\}\}
=
\{\{x\,y\,u\}\,v\,w\}
-
\{u\,\{y\,x\,v\}\,w\}
+
\{u\,v\,\{x\,y\,w\}\},
\qquad
\forall\, x,y,u,v,w \in V.
\]

Associated to this structure are the operators
\[
D(x,y)z := \{x\,y\,z\}, 
\qquad
Q(x)z := \frac{1}{2}\{x\,z\,x\}.
\]

The \emph{canonical Hermitian trace form} on $V$ is the sesquilinear form
\begin{equation}\label{eq:trace-form}
(x|y) := \tr\bigl(D(x,y)\bigr),
\end{equation}
where $\tr$ denotes the ordinary trace of the endomorphism $D(x,y)\in \mathrm{End}_{\mathbb C}(V)$.

The \emph{Bergman operator} is defined for $x,y \in V$ by
\[
B(x,y) := \mathrm{Id} - D(x,y) + Q(x)Q(y).
\]
For $z \in \Omega$, the operator $B(z,z)$ is Hermitian and positive definite with respect to the trace form \eqref{eq:trace-form}. We use the genus-normalized Bergman Hermitian metric
\begin{equation}\label{eq:bergman-metric-jordan}
h_z(u,v)=\frac{1}{\gamma}\bigl(B(z,z)^{-1}u \,\big|\, v\bigr),
\qquad 
u,v\in V\simeq T^{1,0}_z\Omega.
\end{equation}
Here $h_z$ is viewed as a sesquilinear form on $T^{1,0}_z\Omega \times T^{1,0}_z\Omega$, related to the $\mathbb{C}$-bilinear convention of  Section \ref{Preliminaries on complex geometry} by the identification $T^{0,1}_z\Omega\cong\overline{T^{1,0}_z\Omega}$. From now on, $h$ will always refer to the sesquilinear convention.
The metric $h$ is K\"ahler and coincides at $z=0$ with $\gamma^{-1}$ times the Hermitian trace form.

Define
\[
K(z,w) := \mathrm{det} B(z,w)^{-1}.
\]
If $dz$ denotes the Lebesgue measure on $V\simeq\mathbb C^n$, then the Riemannian volume measure associated with this metric is, up to a positive multiplicative constant,
\begin{equation}\label{eq:volume-detB}
d\mu(z) = K(z,z)\, dz.
\end{equation}
Moreover, $\Omega$ is the connected component containing $0$ of the set on which $B(z,z)$ is positive definite,
and $K$ coincides, up to a constant factor, with the Bergman kernel of $\Omega$. The K\"ahler form of the genus-normalized metric is
\[
\omega = \frac{i}{2\gamma} \partial \bar{\partial} \log K(z,z),
\]
so that a global K\"ahler potential is
\[
\Phi(z) = \frac{1}{2\gamma}\log K(z,z) = -\frac12 \log N(z,z),
\]
where the \emph{generic norm} is defined by
\[
N(z,w):=K(z,w)^{-1/\gamma}
\]
and where $\gamma$ is an integer called the \emph{genus} of $\Omega$. 
\medskip

A fundamental structural tool in the Jordan--theoretic description of bounded symmetric domains is the notion of a \emph{Jordan frame}, which yields a spectral (polar) decomposition in $V$.

An element $e\in V$ is called a \emph{tripotent} if $\{e,e,e\}=2e$. For every tripotent $e$, the space $V$ admits the $(\cdot|\cdot)$--orthogonal Peirce decomposition
\[
V=V_0(e)\oplus V_1(e)\oplus V_2(e),
\]
where
\[
V_j(e)=\{\, v \in V : D(e,e)v= j v\,\}, 
\qquad j=0,1,2.
\]
A tripotent $e$ is said to be \emph{primitive} if $\dim V_2(e)=1$, equivalently $V_2(e)=\mathbb{C}e$. Two tripotents $e,f$ are \emph{orthogonal} if $D(e,f)=0$.

A \emph{Jordan frame} is a maximal family
\[
(e_1,\dots,e_r)
\]
of pairwise orthogonal primitive tripotents; the integer $r$ coincides with the rank of $\Omega$. Relative to a Jordan frame, one has the joint Peirce decomposition
\[
V=\bigoplus_{0\le i\le j\le r} V_{ij},
\]
where $V_{00}=0$ and
\[
V_{ij}
=
\left\{ 
v \in V : 
D(e_k,e_k)v= (\delta_{ik}+\delta_{jk}) v,
\ \forall k=1,\dots,r
\right\}.
\]
This decomposition is orthogonal with respect to the canonical Hermitian form.

There exist integers $a,b\in\mathbb{N}$, called the \emph{characteristic multiplicities}, independent of the choice of Jordan frame, such that
\[
\dim V_{ij} = a \quad (1 \le i < j \le r), 
\qquad 
\dim V_{0i} = b \quad (1 \le i  \le r), 
\qquad 
\dim V_{ii} = 1 \quad (1 \le i  \le r),
\]
and 
\begin{align}\label{dimension}
n=\dim V = r + \frac{r(r-1)}{2}a + rb, 
\qquad 
\gamma=(r-1)a+b+2.
\end{align}

\medskip

In the Harish--Chandra realization $\Omega \simeq G/K$, the compact group $K$ is the isotropy subgroup of $G=\mathrm{Aut}(\Omega)^\circ$ at $0$. It identifies with the connected component of the group of linear triple automorphisms of $V$. 
In particular, $K$ acts linearly on $V$ and preserves $\Omega$.

With respect to a Jordan frame, one has a polar decomposition analogous to the Cartan $KAK$ decomposition in $G$ and to the singular value decomposition in matrix theory. More precisely, for every $z\in \Omega$ there exist $k\in K$ and uniquely determined real numbers
\[
1>\lambda_1 \ge \lambda_2 \ge \cdots \ge \lambda_r \ge 0
\]
such that
\begin{align}\label{polar decomposition}
z = k\cdot \left(\sum_{j=1}^r \lambda_j e_j\right).
\end{align}
The numbers $\lambda_j$ are called the \emph{spectral values} of $z$.

From the generic norm $N$, one can define the generic minimal polynomial as a function of $T$ by
\[
    m(T,x,y) = T^r N(T^{-1}x,y).
\]
It is known \cite[(4.15), proof of (4.16)]{Loos1977} that its coefficients are polynomial functions of $(x_i,\bar{y}_j:1 \le i,j\le n)$ and that the squared spectral values $\lambda_i^2$ of $z\in\Omega$ are the roots of $m(T,z,z)$. As a consequence the spectral values are continuous functions of $z$---we refer to \cite[(9.17.4)]{dieudonne1960foundations}.

The \emph{regular set} of $\Omega$ is defined as
\[
\mathrm{reg} (\Omega):=\left\{  z \in \Omega :  1>\lambda_1(z) > \cdots > \lambda_r(z) >0 \right\}.
\]
It is open since the spectral values are continuous functions of $z$. Moreover $\mathrm{reg} (\Omega)$ is dense in $\Omega$ because of the linearity of $(\lambda_{i})_{1 \le i \le r} \mapsto \sum_{j=1}^r \lambda_j e_j$ and the fact that $K$ acts linearly on $V$.
Since the discriminant polynomial $\Delta$ of the generic minimal polynomial is different from zero on $\mathrm{reg}(\Omega)$, applying the implicit function theorem to $m(s,z,z)=0$ yields the smoothness of $\lambda: z \to (\lambda_1(z),\cdots,\lambda_r(z))$ on $\mathrm{reg}(\Omega)$.

\medskip

By $K$--invariance of the trace form one has
\[
(z|z) = \gamma \sum_{j=1}^r \lambda_j^2.
\]
In particular, $(e_i|e_j)=\gamma\delta_{ij}$, whereas the normalized metric satisfies $h_0(e_i,e_j)=\delta_{ij}$.

For $e=\sum_{i=1}^r\lambda_ie_i$, the Bergman operator preserves each Peirce space and acts by scalar multiplication:
\begin{align}
B(e,e)|_{V_{ij}} &= (1-\lambda_{i}^{2})(1-\lambda_{j}^{2})\,\mathrm{Id} \quad (i<j), \nonumber \\
B(e,e)|_{V_{i0}} &= (1-\lambda_{i}^{2})\,\mathrm{Id}, \nonumber \\
B(e,e)|_{V_{ii}} &= (1-\lambda_{i}^{2})^{2}\,\mathrm{Id}. \label{action bergman}
\end{align}
Eventually, these expressions define $B(z,z)$ for any $z \in \Omega$ written in polar form $z = k\cdot \sum_{i=1}^r \lambda_i e_i$ through the relation
\[
    B(k\cdot z, k \cdot z)= k\, B(z,z) \,k^{-1}
    ,
    \qquad \forall k\in K
    .
\]
Thus $B(z,z)$ is block--diagonal with respect to the Peirce decomposition, and
\[
N(z,z)=\prod_{j=1}^r (1-\lambda_j^2).
\]

Finally, from Formula 1.11 in \cite{Faraut-Koranyi}  and the computations on pages 76-77 in \cite{Faraut-Koranyi}  the volume measure \eqref{eq:volume-detB} admits the following polar integration formula:
\begin{align}\label{Polar integration}
\int_\Omega f(z) \, d\mu(z)
=
c_\Omega 
\int_K 
\int_{[0,1]^r} 
f\!\left(k \cdot \sum_{j=1}^r \sqrt{x_j}\, e_j\right)
\, d\nu_{\gamma,a,b}(x)\, dk,
\end{align}
where $dk$ is the normalized Haar measure on $K$, and $\nu_{\gamma,a,b}$ is the \emph{Selberg measure} on $[0,1]^r$ given by
\begin{align}\label{selberg measure}
d\nu_{\gamma,a,b} (x)
=
\prod_{j=1}^r (1-x_j)^{-\gamma}
\prod_{j=1}^r x_j^b
\prod_{1 \le i < j \le r} |x_i-x_j|^a
\; dx_1\cdots dx_r.
\end{align}
Note that it is apparent from the polar integration formula that $\mathrm{reg}(\Omega)$ is a full measure subset of $\Omega$ since the sets of  the form $\{\lambda_i=\lambda_j\}$ and $\{\lambda_i=0 \}$ are contained in hyperplanes of $\mathbb{R}^r$.
The constant $c_\Omega$ is
\[
c_\Omega
=
\pi^n 
\frac{\Gamma\!\left( \frac{a}{2}+1\right)^r}
{\Gamma_\Omega\!\left( \frac{ra}{2}+1\right)
 \Gamma_\Omega\!\left( \frac{n}{r}\right)},
\]
where
\[
\Gamma_\Omega (\nu)
:=
\prod_{j=1}^r 
\Gamma \left( \nu-\frac{j-1}{2}a\right)
\]
is the \emph{Gindikin--Koecher Gamma function}.
\subsection{Brownian motion on $\Omega$}

\begin{theorem}\label{Theorem EDS}
Let $(\beta(t))_{t \ge 0}$ be a complex Brownian motion on $\left(V,\bigl(\cdot \,\big|\,\cdot \bigr)\right) $, i.e. $\beta(t)=\sum_{j=1}^n \beta_j(t) \mathrm{f}_j $ where $(\mathrm{f}_j)_{1 \le j \le n} $ is an orthonormal basis of $\left(V,\bigl(\cdot \,\big|\,\cdot \bigr)\right) $. Our convention is $\beta_j=B_j^{(1)}+iB_j^{(2)}$, where all real and imaginary parts are independent standard real Brownian motions with variance $t$. Then, for every $z \in \Omega$ the following stochastic differential equation in the It\^o sense
\begin{align}\label{EDSbrown}
dX(t)= \sqrt{\gamma}\,B(X(t),X(t))^{1/2} \; d\beta (t), \qquad X(0)=z \in \Omega
\end{align}
has a unique strong solution which is a Brownian motion on $\Omega$ equipped with the genus-normalized Bergman metric.
\end{theorem}

\begin{proof}
First note that by definition $B(z,z)$ can be written in any basis as a Hermitian matrix whose entries are polynomials in the variables $(z_i,\bar{z}_j:1\le i,j\le n )$. It depends continuously on $z$, satisfies $B(0,0)=\mathrm{Id}$, and is invertible throughout the connected domain $\Omega$ \cite[2.11.2]{Loos1977}. Its eigenvalues therefore remain positive, so $B(z,z)^{1/2}$ is well defined.
By local Lipschitzness of $z \to B(z,z)^{1/2}$ in $\Omega$, the stochastic differential equation \eqref{EDSbrown} has a unique solution $(X(t))_{t\ge 0}$ defined up to the exit time $\tau$ of $\Omega$. The inverse tensor of the metric in \eqref{eq:bergman-metric-jordan} is $\gamma B(z,z)$ relative to the trace form. Hence \eqref{fomula-laplacian} shows that $(X(t))_{t\ge 0}$ has generator $\frac{1}{2}\Delta$, where $\Delta$ is the Laplace--Beltrami operator of the genus-normalized Bergman metric. Constant metric rescaling preserves completeness and a uniform Ricci lower bound. Stochastic completeness therefore follows as for the unnormalized Bergman metric; see \cite{Yau1978}. Thus $\tau=+\infty$ almost surely.
\end{proof}

\subsection{Spectral values of the Brownian motion}

Our  goal in this section is to prove that the spectral values of a Brownian motion on $\Omega$ are a diffusion process whose generator is, up to change of variable,  a radial Heckman-Opdam process associated with a root system of type $BC_r$ explicitly determined by the constants $a,b,r$ in the Peirce decomposition. We will need first  the following algebraic lemma.

\begin{lemma}\label{orthogonality}
Let $(e_1,\cdots,e_r)$ be a Jordan frame. Let $\phi$ be a linear map from $(V, (\cdot |\cdot))$ into itself satisfying, for every $x,y,z \in V$,
\begin{align}\label{triple derivation}
\phi( \{ x \, y \, z \})=\{ \phi(x) \, y \, z \}+\{ x \, \phi(y) \, z \}+\{ x \, y \, \phi(z) \}.
\end{align}
Then for every $1\le i \neq j \le r$, one has $(\phi(e_i) | e_j)=0$.
\end{lemma}

\begin{proof}
Let $e$ be a tripotent. Using the identity
\[
2e=\{ e \, e \, e \},
\]
and applying $\phi$ to both sides, relation \eqref{triple derivation} gives
\begin{align*}
2 \phi(e)
&=\{ \phi(e) \, e \, e \}+\{ e \, \phi(e) \, e \}+\{ e \, e \, \phi(e) \} \\
&=2D(e,e)(\phi(e))+2Q(e)(\phi(e)).
\end{align*}

Write $\phi(e)=x_0+x_1+x_2$ according to the Peirce decomposition
\[
V=V_0(e)\oplus V_1(e) \oplus V_2(e).
\]
Since $D(e,e)$ acts on $V_\alpha(e)$ by multiplication by $\alpha$, we obtain
\[
D(e,e)(\phi(e))=x_1+2x_2.
\]

Moreover, by Proposition V.2.1 in \cite[Part V]{Book_complexdomains},
\[
\{ V_\alpha (e) \, V_\beta (e) \, V_\gamma (e) \} \subset V_{\alpha-\beta +\gamma}(e),
\]
with the convention that $V_\alpha (e)=0$ if $\alpha \notin \{0,1,2\}$. It follows that
\[
Q(e)(x_0)=0, \qquad Q(e)(x_1)=0, \qquad Q(e)(x_2) \in V_2(e).
\]
Hence $Q(e)(\phi(e))=Q(e)(x_2) \in V_2(e)$. Substituting into the previous identity yields
\[
2(x_0+x_1+x_2)=2(x_1+2 x_2)+2 Q(e)(x_2).
\]
Comparing Peirce components, we deduce that $x_0=0$. Therefore $\phi(e)$ is orthogonal to $V_0(e)$. Finally, if $(e_1,\dots,e_r)$ is a Jordan frame and $i \neq j$, then $e_j \in V_0(e_i)$, since
\[
D(e_i,e_i)(e_j)=D(e_j,e_i)(e_i)=0.
\]
Thus $(\phi(e_i)\mid e_j)=0$, which completes the proof.
\end{proof}
As we will consider a transformation of the spectral values $\lambda_i(z),i=1,\dots,r$ by an element-wise application of $\arctanh$, we introduce the sets
\[
\mathcal{C}=\left\{ \tau \in \mathbb{R}^r : 0< \tau_r < \cdots < \tau_1  \right\},
\qquad
\bar{\mathcal C}=\left\{ \tau \in \mathbb{R}^r : 0\le \tau_r \le \cdots \le  \tau_1  \right\}.
\]
We consider the diffusion operator on $\bar{\mathcal C}$ given by
\begin{equation}
\label{radial generator}
\begin{aligned}
\mathcal{L}_{k_1,k_2,k_3} = \sum_{i=1}^r \frac{\partial^2}{\partial \tau_i^2}
 + \sum_{i=1}^r& \Bigg( k_3 \coth(\tau_i) + 2k_2 \coth(2\tau_i)\\&+k_1 \sum_{j\neq i} \bigl( \coth(\tau_i - \tau_j) + \coth(\tau_i + \tau_j) \bigr) \Bigg) \frac{\partial}{\partial \tau_i}
\end{aligned}
\end{equation}
This operator belongs to the family of radial Heckman--Opdam operators.
Throughout this text we will always assume $k_1\ge1$ 
 and $k_2+k_3\ge1$.
We refer to \cite{Schapira_Heckman,Schapira_sharp_estimates} for a thorough study of these operators when $k_i\ge0,i\in\{1,2,3\}$ and to Appendix~\ref{sec:Heckman-Opdam_Laplacians} for details about our setting. 

In particular,  $\frac{1}{2}\mathcal{L}_{k_1,k_2,k_3}$ is the generator of a unique diffusion process on $\bar{\mathcal{C}}$ and  there exists a heat kernel denoted $p_t^{k_1,k_2,k_3}(\tau,\eta)$ for this operator on $\bar{\mathcal{C}}$ with respect to the invariant measure
\begin{align}
dm_{k_1,k_2,k_3}(\tau)
&=\;
 \prod_{i=1}^r |\sinh \tau_i|^{k_3}\,|\sinh 2\tau_i|^{k_2}\;
 \prod_{1\le i<j\le r}\Bigl|\sinh(\tau_i-\tau_j)\,\sinh(\tau_i+\tau_j)\Bigr|^{k_1} d\tau
\label{eq:mu-main}.
\end{align}
When all multiplicities are non-negative, some formulas are available \cite{rosler1998, Schapira_sharp_estimates, Schapira_Heckman} for the heat kernel $p_t^{k_1,k_2,k_3}(\tau,\eta)$ of $\frac{1}{2}\mathcal{L}_{k_1,k_2,k_3}$, relying on the theory  \cite{Heckman1987,HeckmanSchlichtkrull1994,Opdam1995,OshimaShimeno2010} of Heckman--Opdam hypergeometric functions.
In Appendix~\ref{sec:Heckman-Opdam_Laplacians}, we extend them to the setting where $k_1\ge1$, \, $k_2\ge0$ and $k_2+k_3\ge1$ using recent results from \cite{honda2024inversion}.

\begin{theorem}\label{spectral diffusion}

Let $(X(t))_{t \ge 0}$ be a Brownian motion on $\Omega$ started from $z\in \Omega$ and let $\lambda_i(t)$, $1 \le i \le r$, be the spectral values of $X(t)$. Then, for every $t >0$, $\mathbb{P}(X(t) \in \mathrm{reg}(\Omega))=1$ and  the process
\[
\tau(t)=(\arctanh(\lambda_1(t)),\ldots,\arctanh(\lambda_r(t))), \quad t \ge 0,
\]
is a diffusion process with generator $\frac{1}{2}\mathcal{L}_{a,1,2 b}$,  a semimartingale and the unique strong solution of a stochastic differential equation
\begin{equation}\label{SDE tau}
\begin{aligned}
d\tau_i(t)={}&d\beta_i(t)
+\Biggl(
b\coth(\tau_i(t))+\coth(2\tau_i(t)) \\
&\quad+\frac{a}{2}\sum_{j \neq i}
\bigl(\coth(\tau_i(t)-\tau_j(t))+\coth(\tau_i(t)+\tau_j(t))\bigr)
\Biggr)dt,
\end{aligned}
\end{equation}
where $\beta(t)$ is a Brownian motion in $\mathbb{R}^r$.

\end{theorem}
\begin{proof}
Let $(X(t))_{t \ge 0}$ be a Brownian motion on $\Omega$ started from $z$, where we first assume that $$z \in \mathrm{reg} (\Omega)=\left\{  z \in \Omega : 1> \lambda_1(z) > \cdots > \lambda_r(z) >0 \right\},$$ and let $\lambda_i(t)$, $1 \le i \le r$, be the spectral values of $X(t)$.  We define the exit time
\begin{align*}
S_{\mathcal C}&=\inf \left\{ t \ge 0 :  X(t) \notin \mathrm{reg} (\Omega)  \right\} \\
&= \inf \left\{ t \ge 0 :  \tau(t)=(\arctanh(\lambda_1(t)),\ldots,\arctanh(\lambda_r(t))) \notin \mathcal{C} \right\}.
\end{align*}
From the construction of the polar decomposition, the map $\lambda: z \mapsto (\lambda_1(z),\cdots,\lambda_r(z))$ is smooth on $\mathrm{reg} (\Omega)$,  therefore the process 
\[
e(t)=\sum_{i=1}^r \lambda_i(t)\, e_i , \; t < S_{\mathcal C},
\]
is a semimartingale. We then claim that there exists a  $K$-valued semimartingale $k(t)$ such that
\[
X(t)=k(t)\cdot e(t), \; t < S_{\mathcal C}.
\]
Indeed, let $M$ be the set of primitive tripotents in $V$ and consider the set of Jordan frames
\[
\mathcal{T}=\left\{ (c_1,\cdots,c_r) \in M^r : D(c_i,c_j)=0, i\neq j \right\}
\]
Then, $\mathcal{T}$ is a compact submanifold of $V^r$ and the map $\mathcal{T} \times \left\{ \lambda \in \mathbb{R}^r:  1>  \lambda_1 > \cdots > \lambda_r >0 \right\} \to \mathrm{reg} (\Omega)$,
\[
((c_1,\cdots,c_r),(\lambda_1,\cdots,\lambda_r)) \to \sum_{i=1}^r \lambda_i \cdot c_i
\]
is a diffeomorphism. Therefore, there exists a $\mathcal{T}$-valued semimartingale $c(t)$ such that
\[
X(t)=\sum_{i=1}^r \lambda_i(t) c_i(t), \; t < S_{\mathcal C}.
\]
Then, the group $K$ acts transitively on $\mathcal{T}$ and the map $\pi:K \to \mathcal{T}, k\to (k\cdot e_1,\cdots,k\cdot e_r)$ is a submersion so we can consider a horizontal lift $k(t)$ of $c(t)$, $t < S_{\mathcal C} $, by $\pi$. The process $k(t)$ is a semimartingale and such that
\[
X(t)=\sum_{i=1}^r \lambda_i(t) k(t)\cdot e_i=k(t) \cdot e(t).
\]

Then, because $K$ acts isometrically on the Riemannian structure of $\Omega$ its action commutes with the Laplacian $\Delta$ and there exists therefore a diffusion operator $\mathcal{D}$ acting on the space of smooth functions $f:\left\{ \lambda \in \mathbb{R}^r:  1>  \lambda_1 > \cdots > \lambda_r >0 \right\} \to \mathbb{R}$  such that
\begin{align}\label{intertwining}
\frac{1}{2}\Delta (f \circ \lambda)=(\mathcal{D} f) \circ \lambda.
\end{align}

From the intertwining \eqref{intertwining}, the process $\lambda(t)$ is a diffusion process with generator $\mathcal{D}$. 

From there, the proof is now  divided into four steps. In Step~1 we compute the quadratic covariations $\langle \tau_i,\tau_j\rangle$, which determine the second-order part of the generator $\mathcal{D}$. In Step~2 we use the polar integration formula \eqref{Polar integration} to compute the drift term of $\mathcal{D}$. In the first two steps we work up to time $S_{\mathcal C}$; in Step~3 we prove that $S_{\mathcal C}=+\infty$ almost surely. Finally, in Step~4 we remove the condition $X(0)=z \in \mathrm{reg}(\Omega)$.

\paragraph{Step 1: Quadratic variation of $\tau$.}

Since $e(t)=k(t)^{-1}\cdot X(t)$, using the Stratonovich differential notation $\circ d$ we obtain
\[
\circ de(t)
=
\circ dk(t)^{-1}\cdot X(t)
+
k(t)^{-1}\cdot \circ dX(t).
\]
From
\[
0=\circ d\bigl(k(t)k(t)^{-1}\bigr)
=
k(t)\circ dk(t)^{-1}
+
\circ dk(t)\,k(t)^{-1},
\]
we deduce
\[
\circ dk(t)^{-1}
=
-\,k(t)^{-1}\circ dk(t)\,k(t)^{-1}.
\]
Hence
\begin{align*}
\circ de(t)
&=
-\,k(t)^{-1}\circ dk(t)\,k(t)^{-1}\cdot X(t)
+
k(t)^{-1}\cdot \circ dX(t) \\
&=
-\,k(t)^{-1}\circ dk(t)\cdot e(t)
+
k(t)^{-1}\cdot dX(t)
+
\text{It\^o correction terms} \\
&=
-\,k(t)^{-1}\circ dk(t)\cdot e(t)
+
k(t)^{-1}\cdot \sqrt{\gamma}\,B(X(t),X(t))^{1/2} d\beta(t)
+
\text{It\^o correction terms} \\
&=
-\,k(t)^{-1}\circ dk(t)\cdot e(t)
+
\sqrt{\gamma}\,B(e(t),e(t))^{1/2}\,k(t)^{-1}\cdot d\beta(t)
+
\text{It\^o correction terms},
\end{align*}
where we used the $K$-equivariance property of the Bergman kernel
\[
B(k\cdot z,k\cdot z)=kB(z,z)k^{-1}.
\]

Since $k(t)\in K$, the term $k(t)^{-1}\circ dk(t)$ takes values in the Lie algebra $\mathfrak{k}$ of $K$. The group $K$ acts linearly and $(\cdot|\cdot)$-isometrically on $V$ by triple automorphisms:
\[
k\cdot\{x\,y\,z\}
=
\{k\cdot x\,k\cdot y\,k\cdot z\},
\qquad
(k\cdot x|k\cdot y)=(x|y).
\]
Thus every $v \in \mathfrak{k}$ acts as a $(\cdot|\cdot)$-skew-symmetric triple derivation:
\[
v\cdot\{x\,y\,z\}
=
\{v\cdot x\,y\,z\}
+
\{x\,v\cdot y\,z\}
+
\{x\,y\,v\cdot z\}.
\]

By Lemma~\ref{orthogonality}, for $i\neq j$,
\[
\bigl(k(t)^{-1}\circ dk(t)\cdot e_i\,\big|\,e_j\bigr)=0,
\]
while for $i=j$ skew-symmetry implies
\[
\mathrm{Re}\,\bigl(k(t)^{-1}\circ dk(t)\cdot e_i\,\big|\,e_i\bigr)=0.
\]
Therefore, relative to the joint Peirce decomposition
\[
V=\bigoplus_{0\le i\le j\le r} V_{ij},
\]
we have
\[
k(t)^{-1}\circ dk(t)\cdot e(t)
\in
\bigoplus_{1\le i\le r} i\mathbb{R}e_i
\;\oplus\!
\bigoplus_{0\le i<j\le r} V_{ij}.
\]
Since
\[
\circ de(t)\in \bigoplus_{1\le i\le r}\mathbb{R}e_i,
\]
the martingale part (and hence the quadratic variation) of $e(t)$ arises solely from the projection of
\[
\sqrt{\gamma}\,B(e(t),e(t))^{1/2}k(t)^{-1}\cdot d\beta(t)
\]
onto $\bigoplus_{1\le i\le r}\mathbb{R}e_i$. Because $K$ acts isometrically on $(V,(\cdot|\cdot))$, $k(t)^{-1}\cdot d\beta(t)=dW(t)$ for a Brownian motion $W$ on $(V,(\cdot|\cdot))$. Decomposing
\[
W(t)=\sum_{0\le i\le j\le r} W_{ij}(t),
\]
and using that $B(e(t),e(t))$ preserves Peirce spaces, the projection of $\sqrt{\gamma}B(e(t),e(t))^{1/2} dW(t)$ onto $\bigoplus_{1\le i\le r}\mathbb{R}e_i$ is the real part of
\[
\sqrt{\gamma}\sum_{i=1}^r (1-\lambda_i(t)^2)\, dW_{ii}(t).
\]
Since $\dim_{\mathbb{C}} V_{ii}=1$ and $(e_i\mid e_i)=\gamma$, we may write $dW_{ii}=\gamma^{-1/2}e_i\,d\zeta_i$ for independent standard complex Brownian motions $\zeta_i$. Taking real parts, we obtain
\[
d\langle \lambda_i,\lambda_j\rangle(t)
=
\delta_{ij}(1-\lambda_i(t)^2)(1-\lambda_j(t)^2)\,dt.
\]
Consequently,
\[
d\langle \tau_i,\tau_j\rangle(t)
=
\frac{1}{(1-\lambda_i(t)^2)(1-\lambda_j(t)^2)}
\,d\langle \lambda_i,\lambda_j\rangle(t)
=
\delta_{ij}\,dt.
\]

\paragraph{Step 2: Computation of the first-order part.}

From Step~1, $\tau(t)$ is a Markov process whose generator  has a second order part given by
\[
\frac{1}{2}\sum_{i=1}^r \frac{\partial^2}{\partial \tau_i^2}.
\]
Using the polar integration formula \eqref{Polar integration}, and the fact that $\Delta$ commutes with the action of $K$ one sees that the invariant and symmetrizing measure for the generator of $\tau(t)$ is given by  the push-forward of the Selberg measure \eqref{selberg measure} under the change of variable $\tau_i=\arctanh(\sqrt{x_i})$. A direct computation gives that this push-forward is given by
\[
J_{a,b}(\tau)\,d\tau,
\]
with
\[
J_{a,b}(\tau)
=
C
\prod_{j=1}^r \sinh(\tau_j)^{2b}
\prod_{j=1}^r \sinh(2\tau_j)
\prod_{1\le i<j\le r}
|\sinh(\tau_i-\tau_j)\sinh(\tau_i+\tau_j)|^a.
\]
where $C$ is a positive constant. A diffusion operator is completely characterized by its second order part and, when it exists, by an invariant and symmetrizing measure. Using that $\tau_i(t)-\tau_j(t)>0$ for $j>i$ and $t < S_{\mathcal C}$, we conclude that the generator of $\tau(t)$ is

\begin{align*}
\mathcal{D}&= \frac{1}{2} \sum_{i=1}^r \left[ \frac{\partial^2}{\partial \tau_i^2} +\frac{\partial \ln J_{a,b}(\tau)}{\partial \tau_i} \frac{\partial}{\partial \tau_i}\right] \\
&  =\frac{1}{2}\sum_{i=1}^r \left[ \frac{\partial^2}{\partial \tau_i^2}+\left( 2b\coth (\tau_i) +2 \coth(2\tau_i)
+a\sum_{j \neq i}\bigl(\coth(\tau_i-\tau_j)+\coth(\tau_i+\tau_j)\bigr)\right)\frac{\partial}{\partial \tau_i}\right].
\end{align*}

\paragraph{Step 3: $S_{\mathcal C}=+\infty$ almost surely.}

From Steps 1 and 2, up to time $S_{\mathcal C}$, the process $\tau(t)$ satisfies
\[
d\tau_i(t)
=
d\beta_i(t)
+
\left(
b\coth(\tau_i(t))
+\coth(2\tau_i(t))
+\frac{a}{2}\sum_{j\neq i}
\bigl(\coth(\tau_i(t)-\tau_j(t))+\coth(\tau_i(t)+\tau_j(t))\bigr)
\right)dt,
\]
where $\beta(t)$ is a Brownian motion in $\mathbb{R}^r$. This system coincides with the radial Heckman--Opdam  stochastic differential equation of type $BC_r$. By \cite[Proposition 4.1]{Schapira_Heckman}, it admits a unique global strong solution for any initial condition in $\mathcal{C}$, 
and according to the discussion above Proposition~4.2 in \cite{Schapira_Heckman}, $k_\alpha + k_{2\alpha}\ge1/2$ implies that the Heckman-Opdam process started at any point in $\mathcal{C}$ never touches the walls almost surely. 
Indeed, this condition written in our setting is $k_{2e_i} + k_{4e_i} \ge1/2$ where $ k_{2e_i}=b, k_{4e_i}=1/2$ for the wall $\tau_r=0$ and $k_{2(e_i-e_j)}=a/2 \ge 1/2$ for the walls $\tau_i=\tau_{i+1}$ where we use the order defining $\mathcal{C}$ to label the walls.
As $a\ge1$ holds for any Hermitian positive Jordan triple system---we refer to the classification in IV.5 in Part~V of \cite{Book_complexdomains}---, the solution hence never exits $\mathcal{C}$ and $S_{\mathcal C}=+\infty$ almost surely.

\paragraph{Step 4: Arbitrary initial conditions.}

Let finally $X(t)$ be a Brownian motion on $\Omega$ started from an arbitrary $z \in \Omega$. The complement $\Omega \setminus \mathrm{reg}(\Omega)$ is the zero set within $\Omega$ of the polynomial $P = \Delta \cdot p_r$, where $\Delta$ is the discriminant of the generic minimal polynomial and $p_r = \prod_{j=1}^{r} \lambda_j^2$ is its constant term. Since $P$ is not identically zero, this set has Riemannian volume zero. Ellipticity and positivity of the heat kernel therefore imply
\[
\mathbb P_z\bigl(X(\varepsilon)\in\mathrm{reg}(\Omega)\bigr)=1,
\qquad \varepsilon>0.
\]
Apply Step~3 after each deterministic time $\varepsilon=1/m$. With probability one, for every $m\ge1$ the path remains in $\mathrm{reg}(\Omega)$ on $[1/m,\infty)$. Intersecting these events shows that $X(t)\in\mathrm{reg}(\Omega)$ for every $t>0$. Thus $\tau(t)$ is a diffusion with generator $\frac{1}{2}\mathcal{L}_{a,1,2 b}$. Finally, by \cite[Proposition 4.1]{Schapira_Heckman}, it is a semimartingale and the unique strong solution of \eqref{SDE tau}.
\end{proof}

\section{The stochastic symplectic area process}

In this section we introduce and study the stochastic area process on any bounded symmetric domain. Our main results are the central limit result Theorem \ref{CLT area} and the characteristic function formula in Theorem \ref{characteristic function area}.

Let $\Omega$ be the bounded symmetric domain described in the previous section. A global K\"ahler potential is
\[
\Phi(z) = \frac{1}{2\gamma}\log \mathrm{det} B (z,z)^{-1}
=-\frac12\log N(z,z), \qquad z \in \Omega .
\]
Among $K$-invariant potentials for the fixed K\"ahler form, it is unique up to an additive constant; in particular,
\[
\Phi(k \cdot z)=\Phi(z), \; z \in \Omega, k \in K.
\]

\subsection{Definition and first properties}

The $d^c$-form associated with the canonical K\"ahler potential is given by
\begin{align*}
\alpha=\frac{1}{2\ic}(\partial-\bar\partial)\Phi.
\end{align*}
The stochastic area process is defined by the Stratonovich line integral
\[
S(t)=\int_{X[0,t]} \alpha
\]
where $X(t)$ is a Brownian motion.

The next lemma that computes the gradient of the spectral values will be used to represent the stochastic area process as a stochastic integral with respect to Brownian motion.

\begin{lemma}\label{grad lambda}
Let $e\in\mathrm{reg}(\Omega)$ be of the form $e=\sum_{i=1}^r \lambda_i e_i$ for a fixed Jordan frame $(e_1,\dots,e_r)$. Then the Riemannian gradients of $\lambda_i,i=1,\dots,r$ with respect to the genus-normalized Bergman metric are
\[
    \nabla\lambda_i(e)=(1-\lambda_i^2(e))^2 e_i.
\]
\end{lemma}

\begin{proof}
Recall that the genus-normalized Hermitian Bergman metric and its associated Riemannian metric are
\[
h_e(u,v)=\frac1\gamma\big(B(e,e)^{-1}u \mid v\big),
\qquad
g_e=\mathbf{Re} \; h_e,
\]
where $(\cdot \mid \cdot)$ denotes the canonical Hermitian trace form.
By the Peirce block decomposition of the Bergman operator, for $1\le i\le r$,
\[
B(e,e)\big|_{V_{ii}}=(1-\lambda_i^2)^2\,\mathrm{Id},
\]
hence
\[
B(e,e)^{-1}e_i=(1-\lambda_i^2)^{-2}e_i .
\]
Therefore, as $e_i,e_j$ are \emph{orthogonal} tripotents we obtain
\[
h_e(e_i,e_j)=g_e(e_i,e_j)
=(1-\lambda_i^2)^{-2}\delta_{ij}.
\]

Since $e$ lies in $\mathrm{reg}(\Omega)$, $\lambda: z \mapsto (\lambda_1(z),\cdots,\lambda_r(z))$ is smooth on some neighborhood of $e$. We claim that
\[
(d\lambda_i)_e(e_j)=\delta_{ij}.
\]
Indeed, consider  $\sigma:(-\varepsilon,\varepsilon) \to \Omega$, a smooth curve with $\sigma(0)=e$ and $\sigma'(0)=e_j$. Then, writing $\sigma(t)=k(t) \cdot e(t)$ with $k(t) \in K$, $k(0)=\mathrm{Id}$ and $e(t)=\sum_{i=1}^r \lambda_i(t) e_i$ we get
\[
e'(t)=(k(t)^{-1} \sigma (t))'=-k(t)^{-1}k'(t)k(t)^{-1} \sigma (t)+k(t)^{-1}\sigma'(t).
\]
Therefore, we have
\[
e'(0)=-k'(0) e+e_j.
\]
Projecting both sides of the equality in the Peirce decomposition and using Lemma \ref{orthogonality} yields, as claimed,
\[
(d\lambda_i)_e(e_j)=\delta_{ij}
\]
because $k'(0) \in \mathfrak{k}$ acts on $V$ as a skew-symmetric derivation satisfying \eqref{triple derivation}; see Step~1 in the proof of Theorem~\ref{spectral diffusion}. Moreover, $K$-invariance of every spectral value implies that $(d\lambda_i)_e$ vanishes on $T_e(K\cdot e)$. At a regular radial point, the polar decomposition and the Peirce decomposition give the $g_e$-orthogonal splitting
\[
T_e\Omega
=
\left(\bigoplus_{j=1}^r\mathbb R e_j\right)
\mathbin{\oplus^{\perp}}
T_e(K\cdot e),
\qquad
T_e(K\cdot e)
=
\bigoplus_{j=1}^r i\mathbb R e_j
\oplus
\bigoplus_{0\le p<q\le r}V_{pq}.
\]
Thus the gradient has no component tangent to the $K$-orbit. By definition of the gradient,
\[
g_e(\nabla\lambda_i (e) , e_j)=\delta_{ij}.
\]
Since the vectors $e_j$, $j=1,\dots,r$, are $g_e$--orthogonal and
$g_e(e_j,e_j)=(1-\lambda_j^2)^{-2}$, we conclude
\[
\nabla\lambda_i(e)=(1-\lambda_i^2(e))^2 e_i .
\]
\end{proof}

Let $f:\Omega \to \mathbb{R}$ be a smooth $K$-invariant function, that is, $f(k\cdot z) = f(z), \;\; \forall k \in K, z\in\Omega$. Then $\nabla f$ is $K$-equivariant:
\[
    \nabla f(k\cdot z)=k\cdot\nabla f(z).
\]
Consequently, to derive $\nabla f$ it suffices to find $\nabla f(e)$ for $e=\sum_{i=1}^{r}\lambda_ie_i$ and write it in $K$-equivariant form.

\begin{corollary}\label{martingale representation for S}
Let $(X(t))_{t \ge 0}$ be a Brownian motion on $\Omega$, and let $\lambda_i(t)$, $1 \le i \le r$, be the spectral values of $X(t)$. Then, there exists a Brownian motion $(\alpha(t))_{t \ge 0}$ in $\mathbb{R}^r$ such that 
\[
S(t) = \frac12 \sum_{i=1}^r \int_0^t \lambda_i(s) d\alpha_i(s), \quad t \ge 0.
\]
\end{corollary}

\begin{proof}

\smallskip
\noindent\textbf{Step 1.}  We first compute the Riemannian gradient of $\Phi$ and prove $\nabla \Phi(z)=B(z,z)^{1/2} z$ for $z \in \Omega$. As $\Phi$ is smooth and $\mathrm{reg}(\Omega)$ is dense in $\Omega$, it suffices to compute it on the former set.
Note that by definition
\[
\Phi(z)=-\frac12 \sum_{j=1}^r \log(1-\lambda_j^2(z))
,
\]
hence, writing $e(z)=\sum_{i=1}^r \lambda_i(z)e_i$ and using Lemma \ref{grad lambda}, we obtain
\begin{align*}
\nabla \Phi \left(e(z)\right)
    =\sum_{i=1}^r \frac{\lambda_i(z)}{1-\lambda_i^2(z)}\nabla\lambda_i(e(z))
    = \sum_{i=1}^r \lambda_i(z)(1-\lambda_i^2(z))e_i
    = B(e(z),e(z))^{1/2} e(z)
,\quad z\in\mathrm{reg}(\Omega)
\end{align*}
whence, using the $K$-equivariance of the Bergman operator 
$
    B(k\cdot z,k\cdot z) = k B(z,z) k^{-1}
$
, we get
\begin{align*}
    \nabla \Phi(z)
    &=
    \nabla \Phi(k(z)\cdot e(z))
    = 
    k(z) \nabla \Phi(e(z))
    = 
    k(z) \cdot B(e(z),e(z))^{1/2} e(z)
    \\&=
    B\left(k(z) \cdot e(z),k(z) \cdot e(z)\right)^{1/2} k(z) \cdot e(z)
    =
    B(z,z)^{1/2} z
\end{align*}
which is continuous as expected and allows one to use a density argument.

\smallskip
\noindent\textbf{Step 2.}  We then prove
\[
\alpha_z(v)=-\frac{1}{2\gamma}\,\mathbf{Im}\bigl(B(z,z)^{-1/2}z\mid v\bigr),
\qquad z\in\Omega,\quad v\in V.
\]
To compute $\alpha$, we use from the proof of Theorem~\ref{martingale property} the fact that under the musical isomorphism the K\"ahler $d^c$-potential form can be expressed as
\[
\alpha^\sharp=\frac{1}{2}J\nabla\Phi.
\]
Therefore, as $g=\mathbf{Re}(h)$ is $J$ invariant, we have
\begin{align*}
\alpha_z (v)
&=\frac{1}{2}g_z(J\nabla\Phi (z),v)
=-\frac{1}{2} g_z(\nabla\Phi (z),Jv)\\
&=-\frac{1}{2\gamma}\mathbf{Re}\bigl(B(z,z)^{-1/2}z\mid Jv\bigr)
=-\frac{1}{2\gamma}\mathbf{Im}\bigl(B(z,z)^{-1/2}z\mid v\bigr).
\end{align*}

\smallskip
\noindent\textbf{Step 3: Computation of $S$.}
From \textbf{Step 2} and Theorem \ref{martingale property} which yields the local martingale property of $S(t)$ one has 
\begin{align*}
S(t)=\int_{X[0,t]} \alpha =-\frac{1}{2\gamma}\int_0^t \mathbf{Im} \left( B(X(s),X(s))^{-1/2} X(s) | dX(s) \right).
\end{align*}
Now, from Theorem \ref{Theorem EDS}, one has $dX(s)=\sqrt\gamma B(X(s),X(s))^{1/2}d \beta(s)$, therefore one obtains by $(\cdot | \cdot)$-symmetry of the Bergman operator
\begin{align*}
S(t)&=-\frac{1}{2\sqrt\gamma}\int_0^t \mathbf{Im} \left( B(X(s),X(s))^{-1/2} X(s) | B(X(s),X(s))^{1/2}d \beta(s) \right) \\
 &=-\frac{1}{2\sqrt\gamma}\int_0^t \mathbf{Im} \left(  X(s) | d\beta(s) \right).
\end{align*}
We now use the polar decomposition to write $X(t)=k(t) \cdot e(t)$ where $k(t) \in K$ and $e(t)=\sum_{i=1}^r \lambda_i(t) e_i$. This yields
\[
S(t)=-\frac{1}{2\sqrt\gamma}\int_0^t \mathbf{Im} \left(  k(s) \cdot e(s) | d\beta(s) \right)=-\frac{1}{2\sqrt\gamma}\int_0^t \mathbf{Im} \left(   e(s) | k(s)^{-1} \cdot d\beta(s) \right).
\]
Since $k$ acts isometrically, the process $W(t):=\int_0^t k(s)^{-1} \cdot d\beta(s)$ is a complex Brownian motion on $(V,(\cdot |\cdot))$. We get then
\[
S(t)=-\frac{1}{2\sqrt\gamma}\int_0^t \sum_{i=1}^r \lambda_i(s) \mathbf{Im} \left(  e_i  | dW(s) \right)
\]
and the conclusion follows thanks to L\'{e}vy's characterization of Brownian motion by defining
\[
\alpha_i(t)=-\frac{1}{\sqrt{\gamma}} \mathbf{Im} \left(  e_i  | W(t) \right).
\]

\end{proof}

\subsection{Central limit theorem for the stochastic area}

The decomposition in Corollary \ref{martingale representation for S} suggests asymptotic Gaussianity, which we now make precise.

\begin{theorem}\label{CLT area}
Let $(X(t))_{t \ge 0}$ be a Brownian motion on $\Omega$ and let $(S(t))_{ t \ge 0}$ be the associated stochastic area process. Then, as $t \to +\infty$, one has the convergence in distribution
\[
\frac{1}{\sqrt{t}}\, S(t)
\;\Longrightarrow\;
\mathcal{N}\!\left(0,\frac r4\right),
\]
where $\mathcal{N}(0,r/4)$ denotes the centered Gaussian distribution with variance $r/4$.
\end{theorem}

\begin{proof}
By Corollary~\ref{martingale representation for S}, the process $S(t)$ is a martingale with quadratic variation
\[
\langle S \rangle (t)
=
\frac14 \sum_{i=1}^r \int_0^t \lambda_i(s)^2\, ds
=
\frac14 \sum_{i=1}^r \int_0^t \tanh(\tau_i(s))^2\, ds.
\]

From Theorem~\ref{spectral diffusion}, the process $\tau=(\tau_1,\dots,\tau_r)$ satisfies the system of stochastic differential equations
\[
\begin{aligned}
d\tau_i(t)={}&d\beta_i(t)
+\Biggl(
b\,\coth(\tau_i(t))+\coth(2\tau_i(t)) \\
&\quad+\frac{a}{2}\sum_{j\neq i}
\bigl(\coth(\tau_i(t)-\tau_j(t))+\coth(\tau_i(t)+\tau_j(t))\bigr)
\Biggr)dt,
\end{aligned}
\]
where $\beta(t)$ is a Brownian motion in $\mathbb{R}^r$. By Proposition~4.2 in \cite{Schapira_Heckman}, one has almost surely
\[
\tau_i(t)\longrightarrow +\infty
\qquad \text{as } t\to+\infty,
\]
for each $i=1,\dots,r$. Consequently,
\[
\tanh(\tau_i(t)) \longrightarrow 1
\qquad \text{a.s.}
\]
and therefore
\[
\frac{1}{t}\langle S \rangle (t)
=
\frac{1}{4t}
\sum_{i=1}^r \int_0^t \tanh(\tau_i(s))^2\, ds
\longrightarrow
\frac r4
\qquad \text{a.s.}
\]

The conclusion now follows from the martingale central limit theorem
(see, for instance, Theorem~4.1 in \cite{MR1792301}), which yields the convergence in distribution
\[
\frac{1}{\sqrt{t}}\, S(t)
\Longrightarrow
\mathcal{N}\!\left(0,\frac r4\right).
\]
\end{proof}
\subsection{Characteristic function of the stochastic area}

While the central limit theorem gives the limit law, we end this section by deriving a closed expression for the characteristic function of the stochastic area $S(t)$ for a fixed $t >0$.

\begin{proposition}\label{skew product area}
Let $(X(t))_{t \ge 0}$ be a Brownian motion on $\Omega$, and let $\lambda_i(t)$, $1 \le i \le r$, denote the spectral values of $X(t)$. Let $(S(t))_{t \ge 0}$ be the associated stochastic area process, and set $\tau_i(t)=\arctanh(\lambda_i(t))$.  Then the process
\[
(\tau(t), S(t))_{t \ge 0}
\]
is a diffusion with generator
\[
\frac{1}{2}\mathcal{L}_{a,1,2 b}
\;+\;
\frac{1}{8}\left(\sum_{i=1}^r \tanh(\tau_i)^2 \right)
\frac{\partial^2}{\partial s^2},
\]
where $\mathcal{L}_{a,1,2 b}$ is given by \eqref{radial generator}.
\end{proposition}

\begin{proof}
Since $(\tau(t))_{t \ge 0}$ is a diffusion with generator $\frac{1}{2}\mathcal{L}_{a,1,2 b}$ and
\[
\langle S \rangle (t)
=
\frac14 \sum_{i=1}^r \int_0^t \tanh(\tau_i(s))^2\,ds,
\]
it suffices to show that for every $i=1,\dots,r$,
\[
\langle \tau_i , S \rangle(t)=0.
\]
Because $\tau_i=\arctanh(\lambda_i)$, this is equivalent to proving
\[
\langle \lambda_i , S \rangle(t)=0.
\]

Since $S(t)=\int_{X[0,t]} \alpha$ where $X$ is a Brownian motion, we have $d\left\langle X^i,X^j \right\rangle(t) = g^{ij}(X(t))dt$ and, by \Cref{spectral diffusion}, almost surely $X(t)\in\mathrm{reg}(\Omega)$ for every $t>0$, on which $\lambda:z\mapsto(\lambda_1(z),\dots,\lambda_r(z))$ is smooth. Hence, we can write
\[
\langle \lambda_i , S \rangle(t)
=
\int_0^t
g_{X(s)}\big(
\nabla \lambda_i (X(s)),
\alpha^\sharp (X(s))
\big)\,ds
=
\frac{1}{2} \int_0^t
g_{X(s)}\big(
\nabla \lambda_i (X(s)),
J \nabla \Phi (X(s))
\big)\,ds.
\]

Thus it remains to prove that on $\mathrm{reg}(\Omega)$ we have
\[
g(\nabla \lambda_i , J \nabla \Phi)=0,
\]
or equivalently,
\[
\mathbf{Im}\big(
h(\nabla \lambda_i , \nabla \Phi)
\big)=0.
\]

By $K$-invariance, it is enough to verify the identity at  elements of the form
\[
e=\sum_{i=1}^r \lambda_i e_i \quad \text{ where } \quad 1>\lambda_1 > \lambda_2 > \dots >\lambda_r >0.
\]
At such a point,
\[
\nabla \lambda_i (e)
=
(1-\lambda_i^2)^2 e_i,
\qquad
\nabla \Phi (e)
=
\sum_{i=1}^r \lambda_i(1-\lambda_i^2)e_i.
\]
Since
\[
h_e(e_i,e_j)
=
(1-\lambda_i^2)^{-2}\delta_{ij},
\]
the Hermitian product
$h_e(\nabla \lambda_i(e), \nabla \Phi(e))$
is real. Therefore its imaginary part vanishes, and the result follows.
\end{proof}

We are now ready for the characteristic function formula.

\begin{theorem}\label{characteristic function area}
    For every $t>0$, $\alpha>0$, and $\tau\in\mathcal C$, choose the continuous radial-bridge version of the conditional expectation. Then
    \begin{align*}
\mathbb{E}\left( e^{i\alpha S(t)} \mid \tau(t)= \tau \right) & =\mathbb{E}\left( e^{-\frac{\alpha^2}{8}\sum_{i=1}^r \int_0^t \tanh(\tau_i(s))^2\,ds } \mid \tau(t)= \tau \right) \\
 &=e^{n\alpha t/2} \frac{p^{a,1+\alpha,2 b-\alpha}_t(\tau(0),\tau)}{p^{a,1,2 b}_t(\tau(0),\tau)}\prod_{i=1}^r \left( 4 \cosh (\tau_i(0)) \cosh (\tau_i)  \right)^{\alpha/2},   
    \end{align*}
    where $p^{k_1,k_2,k_3}_t$ is the heat kernel of the Heckman-Opdam Laplacian $\frac{1}{2}\mathcal{L}_{k_1,k_2,k_3}$ about which details are collected in Appendix \ref{sec:Heckman-Opdam_Laplacians}.
\end{theorem}

\begin{proof}
Let $\alpha>0$. The block-diagonal generator in Proposition \ref{skew product area}, together with weak uniqueness for its martingale problem, allows the joint process to be realized using a one-dimensional Brownian motion $W$ independent of the radial process. Thus
\begin{align*}
\mathbb{E}\left( e^{i\alpha S(t)} \mid \tau(t)= \tau \right)&=\mathbb{E}\left( e^{i\alpha W\left( \frac14\sum_{i=1}^r \int_0^t \tanh(\tau_i(s))^2\,ds \right)} \mid \tau(t)= \tau \right)
\end{align*}
Consequently,
\begin{align}\label{conditional expectation}
\mathbb{E}\left( e^{i\alpha S(t)} \mid \tau(t)= \tau \right)=\mathbb{E}\left( e^{-\frac{\alpha^2}{8}\sum_{i=1}^r \int_0^t \tanh(\tau_i(s))^2\,ds } \mid \tau(t)= \tau \right).
\end{align}
Consider now the local martingale
\[
D_\alpha (t) =\exp \left( \frac\alpha2 \sum_{i=1}^r \int_0^t \tanh(\tau_i(s))\,d\beta_i(s)-\frac{\alpha^2}{8}\sum_{i=1}^r \int_0^t \tanh(\tau_i(s))^2\,ds\right),
\]
where $\beta$ is the Brownian motion in $\mathbb R^r$ that drives the stochastic differential equation for $\tau$:
\begin{equation}\label{SDE tau proof}
\begin{aligned}
d\tau_i(t)={}&d\beta_i(t)
+\Biggl(
b\coth(\tau_i(t))+\coth(2\tau_i(t)) \\
&\quad+\frac{a}{2}\sum_{j\neq i}
\bigl(\coth(\tau_i(t)-\tau_j(t))+\coth(\tau_i(t)+\tau_j(t))\bigr)
\Biggr)dt.
\end{aligned}
\end{equation}
Applying It\^o's formula to $\ln \cosh \tau_i(t)$, Lemma \ref{computation coth} and the identity
\[
\tanh (x)\coth(2x)+\frac{1}{2(\cosh x)^2}=1
\]
one gets
\begin{align*}
 & \sum_{i=1}^r d\ln \cosh \tau_i(t)
 \\
= &\sum_{i=1}^r \left(\tanh(\tau_i(t)) d\tau_i(t)+\frac{dt}{2\cosh \tau_i(t)^2}\right) \\
 =&\sum_{i=1}^r \tanh(\tau_i(t)) d\beta_i(t)+r(b+1)\,dt \\
 &\quad+\frac{a}{2}\sum_{i=1}^r\sum_{j\neq i}\tanh(\tau_i(t))
\bigl(\coth(\tau_i(t)-\tau_j(t))+\coth(\tau_i(t)+\tau_j(t))\bigr)dt \\
=& \left( r(b+1) +\frac{r(r-1)}{2} a \right)dt+\sum_{i=1}^r \tanh(\tau_i(t)) d\beta_i(t) \\
=&n dt +\sum_{i=1}^r \tanh(\tau_i(t)) d\beta_i(t),
\end{align*}
where we used \eqref{dimension} in the last equality. As a consequence, one can rewrite the local martingale $D_\alpha (t)$ as
\begin{align}\label{martingale formula}
D_\alpha (t)=e^{-n\alpha t/2} \exp \left( -\frac{\alpha^2}{8}\sum_{i=1}^r \int_0^t \tanh(\tau_i(s))^2\,ds\right) \prod_{i=1}^r \left( \frac{\cosh \tau_i(t)}{\cosh \tau_i(0) }\right)^{\alpha/2}.
\end{align}
Since $|\tanh x|\le1$, Novikov's criterion shows directly that $D_\alpha$ is a true martingale on every finite time interval. We can therefore consider the probability measure defined by
\[
\mathbb{P}^\alpha_{\mathcal F_t}= D_\alpha (t) \; \mathbb{P}_{\mathcal F_t}
\]
where $\mathcal{F}_t$ is the natural filtration of the Brownian motion $\beta$. From Girsanov's theorem, the process
\[
\beta^{\alpha}_i(t)=\beta_i(t)-\frac\alpha2 \int_0^t \tanh (\tau_i(s)) ds
\]
is a Brownian motion under $\mathbb{P}^\alpha$, and using \eqref{SDE tau proof} together with the identity
\[
\tanh (x)=-\coth (x)+2\coth(2x)
\]
we can write
\begin{align}\label{SDE tau proof Palpha}
 d\tau_i(t) 
=& 
d\beta^\alpha_i(t)
+
\left(
\left(b-\frac\alpha2\right)\coth(\tau_i(t))
+(1+\alpha)\coth(2\tau_i(t)) \right) dt \\
 & 
 +\frac{a}{2}\sum_{j\neq i}
\bigl(\coth(\tau_i(t)-\tau_j(t))+\coth(\tau_i(t)+\tau_j(t))\bigr) dt. \notag
\end{align}
Now, for any bounded Borel function $f$ on $\bar{\mathcal{C}}$, using \cref{martingale formula}, one has
\begin{align*}
\mathbb{E} \left( \exp \left( -\frac{\alpha^2}{8}\sum_{i=1}^r \int_0^t \tanh(\tau_i(s))^2\,ds\right) f(\tau(t))\right)& =e^{n\alpha t/2}\mathbb{E} \left( D_\alpha(t) \prod_{i=1}^r \left( \frac{\cosh \tau_i(0)}{\cosh \tau_i(t) }\right)^{\alpha/2} f(\tau(t))\right) \\
 &=e^{n\alpha t/2}\mathbb{E}_\alpha \left(  \prod_{i=1}^r \left( \frac{\cosh \tau_i(0)}{\cosh \tau_i(t) }\right)^{\alpha/2} f(\tau(t))\right),
\end{align*}
where $\mathbb{E}_\alpha$ denotes the expectation under the probability $\mathbb{P}_\alpha$. Thanks to \eqref{SDE tau proof Palpha}, $\tau(t)$ is under $\mathbb{P}_\alpha$ a diffusion with generator $\frac{1}{2}\mathcal{L}_{a,1+\alpha,2 b-\alpha}$ and therefore
\begin{align*}
 & e^{n\alpha t/2}\mathbb{E}_\alpha \!\left(
  \prod_{i=1}^r \left( \frac{\cosh \tau_i(0)}{\cosh \tau_i(t)} \right)^{\alpha/2}
  f(\tau(t))
\right) \\
=&e^{n\alpha t/2}
\int_{\overline{\mathcal C}}
  \prod_{i=1}^r \left( \frac{\cosh \tau_i(0)}{\cosh \tau_i } \right)^{\alpha/2} 
f(\tau)\, p^{a,1+\alpha,2 b-\alpha}_t(\tau(0),\tau)dm_{a,1+\alpha,2 b-\alpha}(\tau).
\end{align*}
On the other hand we have
\begin{align*}
& \mathbb{E} \left( \exp \left( -\frac{\alpha^2}{8}\sum_{i=1}^r \int_0^t \tanh(\tau_i(s))^2\,ds\right) f(\tau(t))\right) \\
=&\int_{\overline{\mathcal C}} \mathbb{E} \left(   \left. \exp \left( -\frac{\alpha^2}{8}\sum_{i=1}^r \int_0^t \tanh(\tau_i(s))^2\,ds\right)  \right|  \tau(t)=\tau\right) \,f(\tau) p^{a,1,2 b}_t(\tau(0),\tau)dm_{a,1,2 b}(\tau).
\end{align*}
Since this holds for every function $f$ we deduce
\begin{align*}
 & \mathbb{E} \left( \left. \exp \left( -\frac{\alpha^2}{8}\sum_{i=1}^r \int_0^t \tanh(\tau_i(s))^2\,ds\right)  \right| \tau(t) 
 =\tau\right) \\
 =&e^{n\alpha t/2} \prod_{i=1}^r \left( \frac{\cosh \tau_i(0)}{\cosh \tau_i } \right)^{\alpha/2}   \frac{p^{a,1+\alpha,2 b-\alpha}_t(\tau(0),\tau)}{p^{a,1,2 b}_t(\tau(0),\tau)} \frac{dm_{a,1+\alpha,2 b-\alpha}(\tau)}{dm_{a,1,2 b}(\tau)}.
\end{align*}
The density ratio is $\prod_i(2\cosh\tau_i)^\alpha$, which proves the formula for transition-almost every endpoint. The displayed continuous right-hand side specifies the announced bridge version on $\mathcal C$.
\end{proof}

\begin{remark}
    \label{area-relation-heat-kernels}
    Notice that the proof above also holds for $\alpha<0$ with a minor change. According to Proposition \ref{skew product area}, the joint processes $(\tau(t),S(t))_{t\ge0}$ and $(\tau(t),-S(t))_{t\ge0}$ have the same law. As a consequence, for any $\alpha\in\mathbb{R}$ we have proved the relation
    \[
        p^{a,1-\alpha,2 b+\alpha}_t(\tau,\eta)
        =
        e^{n\alpha t}
        p^{a,1+\alpha,2 b-\alpha}_t(\tau,\eta)
        \prod_{i=1}^r \big( 4 \cosh (\tau_i) \cosh (\eta_i)  \big)^{\alpha}
        .
    \]
    In Appendix~\ref{sec:Heckman-Opdam_Laplacians} we prove an integral representation of the heat-kernel under the assumption $k_1\ge1,k_2\ge0$ and $k_3\in\mathbb{R}$ such that $k_2+k_3\ge1$ holds. It follows from the above relation that an integral representation also holds in more generality. 
\end{remark}

\begin{remark}
    The closest antecedent is Theorem~3.3 together with Remark~4.4 in \cite{BaudoinDemniWang2023Stiefel}, which connect the relevant Laplace transform to the connection area in the type~I non-compact Grassmannian. In that case $\gamma=N$, $a=2$, $r=k$, and $b=N-2k$. If $A_{\mathrm{St}}$ denotes the real hyperbolic Stiefel connection area used there, comparison of the connection form with the strict K\"ahler potential fixed above gives
    \[
        A_{\mathrm{St}}=-2S
    \]
    up to the choice of orientation. Thus $2S(t)/\sqrt t\Longrightarrow\mathcal N(0,r)$, in agreement with the Stiefel normalization, whereas the symplectic area $S$ itself has limiting variance $r/4$.

    In rank one, $a=2$, $b=N-2$, and $n=N-1$. Integrating Theorem~\ref{characteristic function area} against the endpoint law gives
    \begin{align*}
        \mathbb{E}\left[e^{i\alpha S(t)}\right]
        =e^{(N-1)\alpha t/2}
        \cosh\!\left(\tau(0)\right)^{\alpha/2}
        \mathbb{E}_{2,1+\alpha,2(N-2)-\alpha}
        \left[\cosh(\tau(t))^{-\alpha/2}\right].
    \end{align*}
    Here $\mathbb E_{k_1,k_2,k_3}$ denotes expectation for the diffusion generated by $\frac12\mathcal L_{k_1,k_2,k_3}$. Under $\rho=\cosh(2\tau)$, the parameter in \cite[Theorem~3.3]{BaudoinDemniWang2023Stiefel} is $\alpha_{\mathrm{BDW}}=\alpha/4$ for the Fourier transform of $S$ when $\alpha>0$; equivalently, $\alpha_{\mathrm{BDW}}=|\lambda|/2$ for $A_{\mathrm{St}}=\pm2S$ at Fourier frequency $\lambda$.
    Beyond the type~I case, i.e., when $a\neq2$, one cannot use a Vandermonde transform as in \cite{BaudoinDemniWang2023Stiefel} to reduce the problem to $r$ independent processes, and therefore the Karlin--McGregor framework cannot be applied. The Jordan-theoretic approach developed here treats all types uniformly, including the two exceptional domains.
\end{remark}

\section{The Jordan determinant winding process in tube domains}

In this section we introduce and study the Jordan determinant winding process. This requires us to restrict ourselves to the class of tube type bounded symmetric domains. After some preliminaries, our main results are Theorem \ref{characteristic function dc log det} and Corollary \ref{Cauchy limit} which describes the Cauchy-law asymptotic limit.

\subsection{Definition and first properties}

Let $\Omega\subset V$ be as in Section~3, and assume throughout this section
that $\Omega$ is of tube type, equivalently $b=0$. We refer to the monograph \cite{FarautKoranyi1994} for an exposition of  the theory of tube type domains. Fix a Jordan frame $(e_1,\ldots,e_r)$ and write
\[
        e=e_1+\cdots+e_r .
\]
In the tube-type case,  the product
\[
        x\circ_e y=\frac12\{x,e,y\}
\]
makes $V$ into a unital Jordan algebra with unit $e$. We denote by
$\mathrm{det}_{\mathcal J}$ the determinant polynomial of this unital Jordan
algebra, normalized by
\[
        \mathrm{det}_{\mathcal J}(e)=1.
\]
See \cite[Part~V]{Book_complexdomains}.
The most important fact for us is that if
\[
        z=k\cdot\sum_{j=1}^r \lambda_j e_j
\]
 then one has
\[
        |\mathrm{det}_{\mathcal J}(z)|=\prod_{j=1}^r \lambda_j.
\]
We shall write
\[
        \mathcal J(z):=|\mathrm{det}_{\mathcal J}(z)|
        =\prod_{j=1}^r \lambda_j(z),
        \qquad
        \Omega^\star:=\{z\in\Omega:\mathcal J(z)>0\}.
\]
Thus $\Omega^\star=\Omega\setminus \mathcal Z$, where
\[
        \mathcal Z:=\{z\in\Omega:\mathrm{det}_{\mathcal J}(z)=0\}.
\]
The function $\mathcal J$ is smooth on $\Omega^\star$. Writing $\log$ for a local branch of
the complex logarithm, we have
\[
    \log \mathrm{det}_{ \mathcal J} = \log |\mathrm{det}_{\mathcal J}| + i \arg(\mathrm{det}_{\mathcal J})
\]
so that a direct application of the Cauchy-Riemann equation to the holomorphic function $\mathrm{det}_{\mathcal J}$ allows us to define the $d^c$-log-determinant $\theta$ as
\[
    \theta 
    := 
    2 d^c \log \mathcal J
    =
    2 d^c \log |\mathrm{det}_{ \mathcal J}|
    = 
    d \arg (\mathrm{det}_{ \mathcal J})
\]
and $\theta$ defined this way is a non-exact closed form that does not depend on the choice of maximal tripotent $e$ that we made or the choice of the principal branch of the logarithm.

The goal of the section is to study the $d^c$ $\log$-determinant form process which is defined as the Stratonovich line integral
\[
A(t)=\int_{X[0,t]} \theta
\]
where $X(t)$ is a Brownian motion on $\Omega$.
Note that $\theta$ is only defined in the set $\Omega^{\star}$  but that the complement of this set is polar for the Brownian motion $X(t)$. Therefore, $A(t)$ is well-defined if $X(0) \in \Omega^{\star}$, which we will always assume in this section.

\begin{remark}[Geometric interpretation]
    By definition $\theta=d\arg(\mathrm{det}_{\mathcal J})$ is the pullback under 
    $\mathrm{det}_{\mathcal J}:\Omega^\star \to \mathbb{C}^\star$ of the standard angular form 
    $\mathbf{Im} \left( \frac{dz}{z}\right)$ on $\mathbb{C}^\star$. Consequently, $A(t)$ coincides with 
    the continuous lift to $\mathbb{R}$ of the angular displacement
    \[
        \arg \mathrm{det}_{ \mathcal J}(X(t)) - \arg \mathrm{det}_{ \mathcal J}(X(0)) \pmod{2\pi},
    \]
    that is, with the cumulative winding angle of the planar semimartingale 
    $\mathrm{det}_{ \mathcal J}(X(t)),t\ge0$ around the origin in $\mathbb{C}$. In $\Omega$, $A(t)$ 
    can thus be interpreted as the winding of the Brownian path around the 
    complex hypersurface $\mathcal Z$.
    One can also interpret $\theta$ topologically. Polar decomposition in a Jordan frame gives $H^1(\Omega^\star;\mathbb Z)\simeq\mathbb Z$, and the class $[\theta/(2\pi)]$ is its generator (equivalently, it generates the integral lattice inside $H^1_{\mathrm{dR}}(\Omega^\star)$). Thus $\theta$ can be interpreted as an index form as in \cite{baudoin2024}. 
\end{remark}

\begin{remark}\label{tube type ass}
The assumption that $\Omega$ is of tube type is essential for the winding
functional.  Outside the tube-type case, one may still form the
$K$-invariant function $\prod_{j=1}^r\lambda_j$, but it is no longer the
modulus of a holomorphic function on $V$. Accordingly,
$d^c\log\bigl(\prod_{j=1}^r\lambda_j\bigr)$ is not closed in general and
does not define a winding form. This is why the determinant
winding process is restricted to tube-type domains.
\end{remark}

\begin{lemma}\label{martingale representation for A}
Let $(X(t))_{t \ge 0}$ be a Brownian motion on $\Omega$, and let $\lambda_i(t)$, $1 \le i \le r$, be the spectral values of $X(t)$. Then, there exists a Brownian motion $(\alpha(t))_{t \ge 0}$ in $\mathbb{R}^r$ such that 
\[
    A(t) = 
        \sum_{i=1}^{r} 
        \int_{0}^{t} \frac{1-\lambda_i^2(s)}{\lambda_i(s)} d\alpha_i(s)
        .
\]
\end{lemma}

\begin{proof}
    We follow a line of reasoning similar to that used in the proof of Corollary \ref{martingale representation for S}.
    
\smallskip
\noindent\textbf{Step 1: Computation of $\nabla \log | \mathrm{det}_{\mathcal J} |$.}
As $\log |\mathrm{det}_{\mathcal J}|$ is smooth on $\Omega^\star$ and
$\mathrm{reg}(\Omega)$ is dense in $\Omega^\star$, it suffices to compute
$\nabla \log | \mathrm{det}_{ \mathcal J} |$ on the regular set.
    Note that by definition
    \[
    \log | \mathrm{det}_{ \mathcal J} |(z)= \sum_{j=1}^r \log \lambda_j(z)
    ,
    \]
    hence, writing $e(z)=\sum_{i=1}^r \lambda_i(z)e_i$ and using \cref{grad lambda}, we obtain
    \[
    \nabla \log |\mathrm{det}_{ \mathcal J}| \left(e(z)\right)
        =
        \sum_{i=1}^{r} \frac{1}{\lambda_i(z)}\nabla\lambda_i\left(e(z)\right)
        =
         \sum_{i=1}^{r} \frac{(1-\lambda_i^2(z))^2}{\lambda_i(z)}e_i
         .
    \]
    For $e=\sum_{i=1}^r \lambda_ie_i$ we define the operator $G$ by
    \begin{align*}
        G(e,e)_{\mid V_{ij}}&= \frac{(1-\lambda_i^2)(1-\lambda_j^2)}{\lambda_i \lambda_j}\Id, \quad &&(i<j)
        \\
        G(e,e)_{\mid V_{i0}}&= \frac{(1-\lambda_i^2)}{\lambda_i}\Id,
        \\
        G(e,e)_{\mid V_{ii}}&= \frac{(1-\lambda_i^2)^2}{\lambda_i^2}\Id,
        \\
        G(k\cdot e,k\cdot e)&= k\,G(e,e)\,k^{-1}, \quad &&\forall k \in K.
    \end{align*}
    As a consequence we can write
    \[
        \nabla \log |\mathrm{det}_{ \mathcal J}| \left(e(z)\right) = G\left(e(z), e(z)\right)e(z).
    \]
    Since $\log |\mathrm{det}_{ \mathcal J}|$ is $K$-invariant, i.e. $\log |\mathrm{det}_{ \mathcal J}|(k\cdot z)=\log |\mathrm{det}_{ \mathcal J}|(z)$, we finally get
    \[
        \nabla \log |\mathrm{det}_{ \mathcal J} |(z)
        =
        k(z) \nabla \log |\mathrm{det}_{ \mathcal J} |\left(e(z)\right)
        =  
        G(z,z)z
    \]
    for $z\in\Omega^\star$, where $k(z)\in K$ is such that
    $z=k(z)\cdot e(z)$.

    \smallskip
\noindent\textbf{Step 2: Computation of $\theta_z(v)$.} From the proof of Theorem \ref{martingale property} we know that 
    \[
        \theta^{\#}
        =
        J \nabla \log |\mathrm{det}_{ \mathcal J}|
    \].
    As $g= \mathbf{Re}(h)$ is $J$ invariant we can write
    \begin{align*}
        \theta_z(v)
        =
        g_z\left(J\nabla \log | \mathrm{det}_{ \mathcal J}|(z),v\right)
        =
        - \mathbf{Re}\,h_z\left(\nabla \log |\mathrm{det}_{ \mathcal J}|(z),Jv\right)
        =
        - \frac1\gamma\mathbf{Im}\left(
            B(z,z)^{-1} G(z,z)z
            \mid
            v
        \right)
        .
    \end{align*}
    Note that for $z=k\cdot e$ where $e=\sum_{i=1}^r\lambda_i e_i$ we have
    \[
        B(z,z)^{-1}G(z,z)\,z
        =
        k\,B\left(e,e\right)^{-1}k^{-1}\, k\,G\left( e, e\right) k^{-1}\,(k\cdot e) 
        = 
        k\,B\left(e,e\right)^{-1}\,G\left( e, e\right) e
        =
        \sum_{i=1}^r \frac{1}{\lambda_i}k\cdot e_i
    \]
    so that finally
    \[
    \theta_z(v) 
        = 
        -\frac{1}{\gamma}
        \sum_{i=1}^{r}
        \frac{1}{\lambda_i(z)}
        \mathbf{Im}
        \left(
            k(z) \cdot e_i \mid v
        \right)
        .
    \]
    \smallskip
\noindent\textbf{Step 3: Computation of $A$.} Recall that we write $X(t)=k(t)\cdot e(t)$ where $k(t)\in K,e(t)= \sum_{i=1}^r \lambda_i(t)e_i$ the polar decomposition of $X(t)$.
    Applying Theorem \ref{martingale property}, we get that $A$ is a local martingale and therefore
    \[
        A(t)
        =
        \int_{X[0,t]} \theta
        =
        -\frac{1}{\gamma} \int_{0}^{t} 
        \sum_{i=1}^{r}
        \frac{1}{\lambda_i(s)}
        \mathbf{Im}\left(
            k(s)\cdot e_i \mid dX(s)
        \right)
        .
    \]
    From Theorem \ref{Theorem EDS}, one has $dX(s)=\sqrt\gamma B(X(s),X(s))^{1/2}d \beta(s)$, therefore one obtains by $(\cdot | \cdot)$-symmetry of the Bergman operator

    \begin{align*}
        A(t) 
        &=
        -\frac{1}{\sqrt\gamma} \int_{0}^{t} 
        \sum_{i=1}^{r}
        \frac{1}{\lambda_i(s)}
        \mathbf{Im}\left(
            k(s)\cdot e_i \mid B(X(s),X(s))^{1/2}d \beta(s)
        \right)
        \\
        &=
        -\frac{1}{\sqrt\gamma} \int_{0}^{t} 
        \sum_{i=1}^{r}
        \frac{1-\lambda_i^2(s)}{\lambda_i(s)}
        \mathbf{Im}\left(
            e_i \mid k^{-1}(s) \cdot d\beta (s)
        \right) 
    \end{align*}
    where we used that $k$ acts isometrically. This last fact also yields that the process $W(t):=\int_0^t k(s)^{-1} \cdot d\beta(s)$ is a complex Brownian motion on $(V,(\cdot |\cdot))$. We get then
    \[
        A(t)
        =
        -\frac{1}{\sqrt\gamma} \int_{0}^{t} 
        \sum_{i=1}^{r}
        \frac{1-\lambda_i^2(s)}{\lambda_i(s)}
        \mathbf{Im}\left(
            e_i \mid dW (s)
        \right)
    \]
    and the conclusion follows thanks to L\'{e}vy's characterization of Brownian motion by defining
    \[
    \alpha_i(t)=-\frac{1}{\sqrt{\gamma}} \mathbf{Im} \left(  e_i  | W(t) \right).
    \]
\end{proof}

By contrast to the stochastic area process, the representation of the winding process involves integrands $(1-\lambda_i^2)/\lambda_i$ blowing up near the boundaries of the Weyl chamber. This suggests that the Gaussian martingale-CLT argument used for the area process is not appropriate here.

\subsection{Characteristic function of the Jordan determinant winding process}

Relying on the representation of the winding process $A$ derived in the last subsection, we now derive a formula for the characteristic function of $A(t)$ for any $t\ge0$.

\begin{proposition}\label{skew product log determinant}
Let $(X(t))_{t \ge 0}$ be a Brownian motion on $\Omega$, and let $\lambda_i(t)$, $1 \le i \le r$, denote the spectral values of $X(t)$. Let $(A(t))_{t \ge 0}$ be the associated $d^c$ $\log$-determinant form process, and set $\tau_i(t)=\arctanh(\lambda_i(t))$.  Then the process
\[
(\tau(t), A(t))_{t \ge 0}
\]
is a diffusion with generator
\[
\frac{1}{2}\mathcal{L}_{a,1,0}
\;+\;
2
\sum_{i=1}^{r}
    \sinh
    \left(
    2\tau_i
    \right)^{-2} \frac{\partial^2}{\partial s^2}
.
\]
Here $\mathcal{L}_{a,1,0}$ is given by \eqref{radial generator}.

\end{proposition}

\begin{proof}
    Using Lemma \ref{martingale representation for A}, a simple computation yields
    \[
        \langle A \rangle(t)
        =
        \sum_{i=1}^{r} \int_{0}^{t}
            \left(
                \frac{1-\lambda^2_i(s)}{\lambda_i(s)}
            \right)^2
        ds
        .
    \]
    Thanks to the identity
    \[
        2\coth(2t)
        = \tanh(t) + \coth(t)
    \]
    one can derive
    \[
        \frac{1-\tanh^2(t)}{\tanh(t)} 
        = 2 \left( 
            \coth(t) - \coth(2t)\right)
    \]
    and finally
    \[
        \left( \frac{1-\tanh^2(t)}{\tanh(t)}  \right)^2 = 4 \coth^2(2t) -4
        =
        4
        \sinh
            \left(
            2t
            \right)^{-2}
    \]
    whence
    \[
         \langle A \rangle(t)
        =
        4
        \sum_{i=1}^{r} \int_{0}^{t}
            \sinh
            \left(
            2\tau_i
            \right)^{-2}
        ds
        .
    \]
    To conclude, it suffices to prove that for every $i=1,\dots,r$,
    \[
        \langle
            \tau_i
            ,
            A
        \rangle
        =
        0
    \].
    One can do so by reproducing the proof of Proposition \ref{skew product area} with minor changes. Indeed, the argument presented there still holds since for $e= \sum_{i=1}^{r}\lambda_i\,e_i$ where $\lambda_i>0,i=1,\dots,r$, $\nabla \log |\mathrm{det}_{ \mathcal J}|(e)$ like $\nabla\Phi(e)$ belongs to the span of $(e_i,i=1,\dots,r)$.
\end{proof}

We are now ready for the characteristic function formula.

\begin{theorem}
\label{characteristic function dc log det}
For every \(t>0\), \(\alpha>0\), and \(\tau\in\mathcal C\), choose the continuous radial-bridge version of the conditional expectation. Then
\[
\mathbb E\left(e^{i\alpha A(t)}\mid \tau(t)=\tau\right)
=
\mathcal A_\alpha(t,\tau)
=
\mathcal B_\alpha(t,\tau),
\]
where
\begin{align*}
\mathcal A_\alpha(t,\tau)
&=
e^{2\alpha^2rt}
\mathbb E\left(
\left.
\exp\left\{
-2\alpha^2
\sum_{i=1}^r\int_0^t \coth^2(2\tau_i(s))\,ds
\right\}
\right| \tau(t)=\tau
\right)
\\
&=
e^{
\alpha r\left(2+2\alpha+a(r-1)\right)t
}
\prod_{i=1}^r
\left(
\sinh(2\tau_i(0))\sinh(2\tau_i)
\right)^{\alpha}
\frac{
p_t^{a,\,1+2\alpha,\,0}(\tau(0),\tau)
}{
p_t^{a,\,1,\,0}(\tau(0),\tau)
},
\end{align*}
and
\begin{align*}
\mathcal B_\alpha(t,\tau)
&=
\mathbb E\left(
\left.
\exp\left\{
-2\alpha^2
\sum_{i=1}^r\int_0^t \sinh^{-2}(2\tau_i(s))\,ds
\right\}
\right| \tau(t)=\tau
\right)
\\
&=
2^{-2\alpha r}
\prod_{i=1}^r
\left(
\tanh(\tau_i(0))\tanh(\tau_i)
\right)^{\alpha}
\frac{
p_t^{a,\,1-2\alpha,\,4\alpha}(\tau(0),\tau)
}{
p_t^{a,\,1,\,0}(\tau(0),\tau)
}.
\end{align*}
Here \(p_t^{k_1,k_2,k_3}\) denotes the heat kernel of
\(\frac12\mathcal L_{k_1,k_2,k_3}\); see Appendix \ref{sec:Heckman-Opdam_Laplacians}.
\end{theorem}

\begin{proof}
We split the proof into three steps. First we reduce the characteristic
function of \(A(t)\) to exponential functionals of the radial process. Then we
introduce two Girsanov transforms. Finally, we identify the resulting
conditional expectations by comparing heat-kernel densities.

\medskip

\noindent\textbf{Step 1: Reduction to radial exponential additive functionals.}
By Proposition \ref{skew product log determinant}, the joint process \((\tau(t),A(t))\)
has no mixed second-order terms between the radial variables and the
\(A\)-coordinate. Weak uniqueness for the corresponding martingale problem
therefore allows a realization with a one-dimensional Brownian motion $W$
independent of the process $\tau(s),s\ge0$ such that
\[
    \mathbb{E}\left( e^{i\alpha A(t)} \mid \tau(t)= \tau \right)
        =\mathbb{E}\left( e^{i\alpha 
        W\left( 
            \langle A \rangle(t)
        \right)}
        \mid \tau(t)= \tau \right)
        =
            \mathbb{E}\left( 
            e^{ 
             -\frac{\alpha^2}{2} \langle A \rangle(t)
        }
        \mid \tau(t)= \tau \right)
    \]
and hence
\[
\mathbb E\left(e^{i\alpha A(t)}\mid \tau(t)=\tau\right)
=
\mathbb E\left(
\left.
\exp\left\{
-2\alpha^2
\sum_{i=1}^r\int_0^t \sinh^{-2}(2\tau_i(s))\,ds
\right\}
\right| \tau(t)=\tau
\right).
\]
Using
\[
\sinh^{-2}(2x)=\coth^2(2x)-1,
\]
we also obtain
\[
\mathbb E\left(e^{i\alpha A(t)}\mid \tau(t)=\tau\right)
=
e^{2\alpha^2rt}
\mathbb E\left(
\left.
\exp\left\{
-2\alpha^2
\sum_{i=1}^r\int_0^t \coth^2(2\tau_i(s))\,ds
\right\}
\right| \tau(t)=\tau
\right).
\]
This proves the first identities in the definitions of
\(\mathcal A_\alpha\) and \(\mathcal B_\alpha\).

\medskip

\noindent
\textbf{Step 2: The two Girsanov transforms.}
In the tube-type case, the radial process \(\tau(t)\) has generator
\(\frac12\mathcal L_{a,1,0}\). Equivalently, for \(i=1,\dots,r\),
\begin{align}
\label{SDE tau proof 2}
d\tau_i(t)
&=
d\beta_i(t)
+
\left(
\coth(2\tau_i(t))
+
\frac{a}{2}
\sum_{j\neq i}
\bigl(
\coth(\tau_i(t)-\tau_j(t))
+
\coth(\tau_i(t)+\tau_j(t))
\bigr)
\right)dt,
\end{align}
where \(\beta\) is an \(\mathbb R^r\)-valued Brownian motion. Define the exponential local martingales
\[
D_\alpha(t)
=
\exp\left(
2\alpha
\sum_{i=1}^r
\int_0^t
\coth(2\tau_i(s))\,d\beta_i(s)
-
2\alpha^2
\sum_{i=1}^r
\int_0^t
\coth^2(2\tau_i(s))\,ds
\right),
\]
and
\[
\widetilde D_\alpha(t)
=
\exp\left(
2\alpha
\sum_{i=1}^r
\int_0^t
\frac{d\beta_i(s)}{\sinh(2\tau_i(s))}
-
2\alpha^2
\sum_{i=1}^r
\int_0^t
\sinh^{-2}(2\tau_i(s))\,ds
\right).
\]

We next rewrite these martingales in a form suitable for the heat-kernel
comparison. By It\^o's formula and Lemma \ref{computation coth bis},
\[
\sum_{i=1}^r d\log\sinh(2\tau_i(t))
=
r\bigl(2+a(r-1)\bigr)\,dt
+
2\sum_{i=1}^r
\coth(2\tau_i(t))\,d\beta_i(t).
\]
Similarly,
\[
\sum_{i=1}^r d\log\tanh(\tau_i(t))
=
2\sum_{i=1}^r
\frac{d\beta_i(t)}{\sinh(2\tau_i(t))}.
\]
Consequently,
\begin{align*}
D_\alpha(t)
&=
e^{
-\alpha
r\bigl(2+a(r-1)\bigr)t
}
\exp\left(
-2\alpha^2
\sum_{i=1}^r
\int_0^t
\coth^2(2\tau_i(s))\,ds
\right)
\\
&\qquad\qquad\qquad\qquad\times
\prod_{i=1}^r
\left(
\frac{\sinh(2\tau_i(t))}
     {\sinh(2\tau_i(0))}
\right)^{\alpha},
\end{align*}
and
\[
\widetilde D_\alpha(t)
=
\exp\left(
-2\alpha^2
\sum_{i=1}^r
\int_0^t
\sinh^{-2}(2\tau_i(s))\,ds
\right)
\prod_{i=1}^r
\left(
\frac{\tanh(\tau_i(t))}
     {\tanh(\tau_i(0))}
\right)^{\alpha}.
\]

We now check that these local martingales are true martingales. Since
\(X(0)\in\Omega^\star\), all quantities
\(\sinh(2\tau_i(0))\) and \(\tanh(\tau_i(0))\) are nonzero. Hence, for every
fixed \(t>0\), there exist constants \(C,\widetilde C>0\) such that, for
\(0\le s\le t\),
\[
D_\alpha(s)
\le
C\exp\left(
2\alpha r\tau_1(s)
\right),
\qquad
\widetilde D_\alpha(s)
\le
\widetilde C.
\]
Choose $p>1$. Since $e^{2\alpha r\tau_1(s)}$ is a nonnegative
submartingale, Doob's $L^p$ maximal inequality and
\cref{exponential integrability} give
\[
\mathbb E\left[\sup_{0\le s\le t}e^{2\alpha r\tau_1(s)}\right]
\le \frac p{p-1}
\left(\mathbb E e^{2p\alpha r\tau_1(t)}\right)^{1/p}<\infty.
\]
Thus
\[
\mathbb E\left[\sup_{0\le s\le t}D_\alpha(s)\right]<+\infty,
\qquad
\mathbb E\left[\sup_{0\le s\le t}\widetilde D_\alpha(s)\right]<+\infty.
\]
Therefore \(D_\alpha\) and \(\widetilde D_\alpha\) are true martingales. Let \(\mathbb P_\alpha\) and \(\widetilde{\mathbb P}_\alpha\) be the probability
measures defined on \(\mathcal F_t\) by
\[
\frac{d\mathbb P_\alpha}{d\mathbb P}\bigg|_{\mathcal F_t}
=
D_\alpha(t),
\qquad
\frac{d\widetilde{\mathbb P}_\alpha}{d\mathbb P}\bigg|_{\mathcal F_t}
=
\widetilde D_\alpha(t).
\]
By Girsanov's theorem,
\[
\beta_i^\alpha(t)
=
\beta_i(t)
-
2\alpha
\int_0^t
\coth(2\tau_i(s))\,ds,
\qquad i=1,\dots,r,
\]
is an \(\mathbb R^r\)-Brownian motion under \(\mathbb P_\alpha\), while
\[
\widetilde\beta_i^\alpha(t)
=
\beta_i(t)
-
2\alpha
\int_0^t
\frac{ds}{\sinh(2\tau_i(s))},
\qquad i=1,\dots,r,
\]
is an \(\mathbb R^r\)-Brownian motion under
\(\widetilde{\mathbb P}_\alpha\). It follows from \eqref{SDE tau proof 2}
that, under \(\mathbb P_\alpha\), the process \(\tau\) has generator
\[
\frac12\mathcal L_{a,\,1+2\alpha,\,0}.
\]
Using the identity
\[
\coth u-\coth(2u)=\frac{1}{\sinh(2u)},
\]
we also see that, under \(\widetilde{\mathbb P}_\alpha\), the process \(\tau\)
has generator
\[
\frac12\mathcal L_{a,\,1-2\alpha,\,4\alpha}.
\]

\medskip

\noindent
\textbf{Step 3: Identification by heat-kernel comparison.}
Let \(f\) be a bounded Borel function on \(\bar{\mathcal C}\). From the
previous display for \(D_\alpha(t)\), we get
\begin{align}
\label{condition f winding first}
&\mathbb E\left(
\exp\left\{
-2\alpha^2
\sum_{i=1}^r
\int_0^t\coth^2(2\tau_i(s))\,ds
\right\}
f(\tau(t))
\right)
\nonumber
\\
&\quad =
e^{
\alpha
r\bigl(2+a(r-1)\bigr)t
}
\mathbb E_\alpha\left[
\prod_{i=1}^r
\left(
\frac{\sinh(2\tau_i(0))}
     {\sinh(2\tau_i(t))}
\right)^{\alpha}
f(\tau(t))
\right].
\end{align}
Similarly, from the expression for \(\widetilde D_\alpha(t)\),
\begin{align}
\label{condition f winding second}
&\mathbb E\left(
\exp\left\{
-2\alpha^2
\sum_{i=1}^r
\int_0^t\sinh^{-2}(2\tau_i(s))\,ds
\right\}
f(\tau(t))
\right)
\nonumber
\\
&\quad =
\widetilde{\mathbb E}_\alpha\left[
\prod_{i=1}^r
\left(
\frac{\tanh(\tau_i(0))}
     {\tanh(\tau_i(t))}
\right)^{\alpha}
f(\tau(t))
\right].
\end{align}

We now apply the heat-kernel representation to
\eqref{condition f winding first}. Since under \(\mathbb P_\alpha\) the
radial process has generator
\(\frac12\mathcal L_{a,\,1+2\alpha,\,0}\), the right-hand side of
\eqref{condition f winding first} equals
\begin{align*}
& e^{
\alpha
r\bigl(2+a(r-1)\bigr)t
}
\int_{\bar{\mathcal C}}
\prod_{i=1}^r
\left(
\frac{\sinh(2\tau_i(0))}
     {\sinh(2\eta_i)}
\right)^{\alpha}
f(\eta)
p_t^{a,\,1+2\alpha,\,0}(\tau(0),\eta)
\,dm_{a,\,1+2\alpha,\,0}(\eta).
\end{align*}
On the other hand, by conditioning with respect to \(\tau(t)\), the left-hand
side of \eqref{condition f winding first} is
\begin{align*}
\int_{\bar{\mathcal C}}
\mathbb E\left(
\left.
\exp\left\{
-2\alpha^2
\sum_{i=1}^r
\int_0^t\coth^2(2\tau_i(s))\,ds
\right\}
\right| \tau(t)=\eta
\right)
f(\eta)
p_t^{a,\,1,\,0}(\tau(0),\eta)
\,dm_{a,\,1,\,0}(\eta).
\end{align*}
Since this holds for every bounded Borel function \(f\), we identify the two
densities. Using the explicit expression of \(dm_{k_1,k_2,k_3}\) in
\eqref{eq:mu-main}, we obtain
\begin{align*}
&\mathbb E\left(
\left.
\exp\left\{
-2\alpha^2
\sum_{i=1}^r
\int_0^t\coth^2(2\tau_i(s))\,ds
\right\}
\right| \tau(t)=\tau
\right)
\\
&\quad =
e^{
\alpha
r\bigl(2+a(r-1)\bigr)t
}
\prod_{i=1}^r
\left(
\sinh(2\tau_i(0))\sinh(2\tau_i)
\right)^{\alpha}
\frac{
p_t^{a,\,1+2\alpha,\,0}(\tau(0),\tau)
}{
p_t^{a,\,1,\,0}(\tau(0),\tau)
}.
\end{align*}
Multiplying by \(e^{2\alpha^2rt}\) gives the stated expression for
\(\mathcal A_\alpha(t,\tau)\).

The same argument applied to \eqref{condition f winding second}, now using
the heat kernel of
\[
\frac12\mathcal L_{a,\,1-2\alpha,\,4\alpha},
\]
gives
\[
\mathcal B_\alpha(t,\tau)
=
2^{-2\alpha r}
\prod_{i=1}^r
\left(
\tanh(\tau_i(0))\tanh(\tau_i)
\right)^{\alpha}
\frac{
p_t^{a,\,1-2\alpha,\,4\alpha}(\tau(0),\tau)
}{
p_t^{a,\,1,\,0}(\tau(0),\tau)
}.
\]
The density comparison first proves these identities for transition-almost
every endpoint; the displayed continuous right-hand sides specify the bridge
versions on $\mathcal C$. This completes the proof.
\end{proof}

\begin{remark}
    As a byproduct, note that we have derived another relation between heat kernels of the Heckman-Opdam Laplacian with different multiplicity functions: for any $\alpha>0$,
    \[
        p^{a,1+2\alpha,0}_t(\tau,\eta)
        =
        2^{-4\alpha r}
        e^{-\alpha r(2+2\alpha+a(r-1))t}
        p^{a,1-2\alpha,4\alpha}_t(\tau,\eta)
        \prod_{i=1}^{r} 
        \bigl(\cosh(\tau_i)\cosh(\eta_i)\bigr)^{-2\alpha}
        .
    \]
    As in \cref{area-relation-heat-kernels}, it follows from this relation that an integral representation of the heat kernel $p^{a,1-2\alpha,4\alpha}_t(\tau,\eta)$ still exists even when the assumption $k_2 = 1-2\alpha\ge 0$ is violated.
\end{remark}

\begin{remark}[Comparison with the rank-one tube type~I case]
    In the rank $r=1$ tube case, the only Hermitian positive Jordan triple that arises is type~I with $m=n=r=1$, that is, $\mathbb{C}H^1$ identified with the Poincar\'{e} disk. It corresponds to the parameters $r=1$, $a=2$, $b=0$, $\gamma=2$.
    In this case, we have $\mathcal J = |\lambda|$ and the Riemannian distance from $0$ to $\lambda$ is given by
    \[
        \tau = \arctanh |\lambda|.
    \]
    The $d^c$ $\log$-determinant form process $A$ studied in this section matches the winding process studied in \cite{BaudoinWang2017} and \cite[Section~1.2.4]{BaudoinBook2024}.
    Taking the Fourier parameter there equal to $\alpha$ in \cite[Eq.~(4.15)]{BaudoinWang2017} yields the same expression for $\mathbb{E}\left(e^{i\alpha A(t)}\right)$ as one can derive from the above theorem:
    \[
    \mathbb{E}\left(e^{i\alpha A(t)}\right)
    =
    \widetilde{\mathbb{E}}_\alpha
    \left[
        \left(
            \frac{\tanh(\tau(0))}{\tanh(\tau(t))}
        \right)^{\alpha}
    \right].
\]
\end{remark}

We end this section with the analogue of Spitzer's Theorem \cite{Spitzer1958} in  non-compact Hermitian symmetric spaces of tube type. For the type~I, rank one case, it already appeared as Theorem~4.2 in \cite{BaudoinWang2017}.

\begin{corollary}\label{Cauchy limit}
    As $t\to\infty$, we have the convergence in distribution
    \[
        A(t)
        \;\Longrightarrow\;
        \mathcal{C}_{\log\frac{1}{\mathcal J(X(0))}}
        ,
    \]
    where $\mathcal{C}_s$ denotes the Cauchy distribution with scale parameter $s$.
\end{corollary}

\begin{proof}

In the proof of Theorem \ref{characteristic function dc log det}, we proved that, for every $\alpha>0$,
\[
    \mathbb{E}\left(e^{i\alpha A(t)}\right)
    =
    \mathbb{E}\left(
        \exp\left\{
            - \frac{\alpha^2}{2}
            \sum_{i=1}^{r}
            \int_{0}^{t}
            \left(
                \frac{1-\tanh^2(\tau_i(s))}{\tanh(\tau_i(s))}
            \right)^2 ds
        \right\}
    \right)
    =
    \widetilde{\mathbb{E}}_\alpha
    \left[
        \prod_{i=1}^{r}
        \left(
            \frac{\tanh(\tau_i(0))}{\tanh(\tau_i(t))}
        \right)^{\alpha}
    \right].
\]
It remains to justify the passage to the limit in the last expectation. Indeed, by the law of large numbers for the Heckman--Opdam process
\cite[Proposition~4.2]{Schapira_Heckman}, we have
\[
    \tau_i(t)\to+\infty,
    \qquad i=1,\dots,r,
    \qquad
    \mathbb P\text{-a.s.}
\]
Since the martingale $\widetilde D_\alpha(t)$ is bounded, it is uniformly integrable. Hence $\widetilde{\mathbb{P}}_\alpha$ is absolutely continuous with respect to $\mathbb P$ on $\mathcal{F}_\infty$, and therefore
\[
    \tau_i(t)\to+\infty,
    \qquad i=1,\dots,r,
    \qquad
    \widetilde{\mathbb P}_\alpha\text{-a.s.}
\]

For convenience, set
\[
    G(t):=
    \prod_{i=1}^{r}
    \left(
        \frac{\tanh(\tau_i(0))}{\tanh(\tau_i(t))}
    \right)^\alpha,
    \qquad
    I(t):=
    \prod_{i=1}^{r}
    \mathbf{1}_{\{\tau_i(t)\ge 1\}}.
\]
Then
\[
    \widetilde{\mathbb{E}}_\alpha(G(t))
    =
    \widetilde{\mathbb{E}}_\alpha(I(t)G(t))
    +
    \widetilde{\mathbb{E}}_\alpha((1-I(t))G(t)).
\]
We first consider the term with $I(t)$. Since $\tau_i(t)\to+\infty$ under $\widetilde{\mathbb P}_\alpha$, we have
\[
    I(t)G(t)
    \longrightarrow
    \prod_{i=1}^{r}\tanh(\tau_i(0))^\alpha,
    \qquad
    \widetilde{\mathbb P}_\alpha\text{-a.s.}
\]
Thus, by dominated convergence,
\[
    \lim_{t\to+\infty}
    \widetilde{\mathbb{E}}_\alpha(I(t)G(t))
    =
    \prod_{i=1}^{r}\tanh(\tau_i(0))^\alpha
    =
    \mathcal J(X(0))^{\alpha}.
\]

To conclude, it suffices to show that the second term vanishes. Using \eqref{condition f winding second}, we obtain
\begin{align*}
    \widetilde{\mathbb{E}}_\alpha((1-I(t))G(t))
    &\le
    1-\mathbb{P}\left(
        \tau_1(t)\ge 1,\ldots,\tau_r(t)\ge 1
    \right).
\end{align*}
Since $\tau_i(t)\to+\infty$ for every $i=1,\dots,r$, $\mathbb P$-a.s., the right-hand side tends to $0$. Hence
\[
    \lim_{t\to+\infty}
    \widetilde{\mathbb{E}}_\alpha((1-I(t))G(t))
    =
    0.
\]
Combining the two estimates gives
\[
    \lim_{t\to+\infty}
    \widetilde{\mathbb{E}}_\alpha
    \left[
        \prod_{i=1}^{r}
        \left(
            \frac{\tanh(\tau_i(0))}{\tanh(\tau_i(t))}
        \right)^{\alpha}
    \right]
    =
    \prod_{i=1}^{r}\tanh(\tau_i(0))^{\alpha}
    =
    \mathcal J(X(0))^{\alpha}.
\]
Therefore, for every $\alpha>0$,
\[
    \lim_{t\to+\infty}
    \mathbb{E}\left(e^{i\alpha A(t)}\right)
    =
    \mathcal J(X(0))^{\alpha}.
\]
Finally, by the equality in law of $A(t)$ and $-A(t)$, this extends to
\[
    \lim_{t\to+\infty}
    \mathbb{E}\left(e^{i\alpha A(t)}\right)
    =
    \mathcal J(X(0))^{|\alpha|},
    \qquad \alpha\in\mathbb R,
\]
which is the characteristic function of the announced Cauchy distribution.
\end{proof}

The scale parameter of the limiting Cauchy distribution is the logarithmic determinant depth of the starting point relative to the hypersurface $\{ \mathrm{det}_{\mathcal J} = 0\}$. It diverges as $X(0)$ approaches the hypersurface, so the weak convergence above is consistent with the heuristic that paths starting close to the hypersurface accumulate more windings before escaping.

\appendix

\section{Heckman-Opdam Laplacians, processes and heat kernels}
\label{sec:Heckman-Opdam_Laplacians}

In this appendix we collect some results on the radial Heckman--Opdam Laplacian of type $BC_r$, its associated diffusion process and heat kernel. While these are classical when all multiplicity parameters are nonnegative, here we extend them to a setting where one multiplicity is allowed to be negative, which is the case relevant to the rest of the paper.

Introduce the notation
\[
\mathcal{C}=\left\{ \tau \in \mathbb{R}^r : 0< \tau_r < \cdots < \tau_1  \right\},
\qquad
\bar{\mathcal C}=\left\{ \tau \in \mathbb{R}^r : 0\le \tau_r \le \cdots \le  \tau_1  \right\},
\]
and consider the diffusion operator on $\bar{\mathcal C}$
\begin{equation}
\label{radial generator bis}
\begin{aligned}
\mathcal{L}_{k_1,k_2,k_3} = \sum_{i=1}^r \frac{\partial^2}{\partial \tau_i^2}
 + \sum_{i=1}^r& \Bigg( k_3 \coth(\tau_i) + 2k_2 \coth(2\tau_i)\\&+k_1 \sum_{j\neq i} \bigl( \coth(\tau_i - \tau_j) + \coth(\tau_i + \tau_j) \bigr) \Bigg) \frac{\partial}{\partial \tau_i}
\end{aligned}
\end{equation}
where we will always assume, when not stated otherwise, that $k_1\ge1$
and $k_2+k_3\ge1$. That is we allow $k_2,k_3$ to take negative values with some constraint on $k_2+k_3$. When no confusion is possible, we will denote by $k$ the triple $(k_1,k_2,k_3)$.

This operator belongs to the family of radial Heckman--Opdam operators. We refer to \cite{Schapira_Heckman,Schapira_sharp_estimates} for a thorough study of these operators and of the associated heat kernels when all multiplicities $k_i$ are nonnegative. Note that in our case, the root system associated with $\mathcal{L}_{k}$ is of type $BC_r$. In the notation of \cite{Schapira_Heckman,OshimaShimeno2010}, $\mathcal{L}_{k}$ is exactly the $BC_r$ Heckman--Opdam hypergeometric Laplacian with root system $\mathcal{R}$ determined by the positive root systems  
\[
    \mathcal{R}_+ = \left\{ 
        2e_i, 4e_i, 2e_p \pm 2e_q|1\le i\le r, 1\le p<q\le r
    \right\}
    ,
\]
whose associated Weyl group is denoted $W$,
and $W$-invariant multiplicity parameters
\[
k_{2e_i\pm 2e_j}=\tfrac{k_1}{2}\quad(\text{for }2e_i\pm 2e_j, i\neq j),\qquad
k_{4e_i}=\tfrac{k_2}{2}\quad(\text{for }4e_i),\qquad
k_{2e_i}=\tfrac{k_3}{2}\quad(\text{for }2e_i).
\]

\subsection{The diffusion process and its associated heat kernel}

When $k_i>0,i\in\{1,2,3\}$ it is known (\cite{Schapira_Heckman}) that $\frac{1}{2}\mathcal{L}_{k}$ is the generator of a unique diffusion process on $\bar{\mathcal{C}}$; we now extend this result to a setting where $k_2$ and $k_3$ are allowed to be negative.

\begin{lemma}
    \label{sol SDE}
    Let $k_1\ge 1$ and $k_2,k_3 \in \mathbb{R}$ be such that $k_2+k_3\ge 1$, then for any $X(0)\in\bar{\mathcal{C}}$ and standard Brownian motion $(\beta_i(t),i=1,\dots,r),t\ge0$ there exists a path-wise unique solution $(X_i(t),i=1,\dots,r),t\ge0$ to the stochastic differential equation 
    \begin{equation}
    \label{SDE}
    \begin{split}
        dX_i(t) = d\beta_i(t) &+ \Bigg( \frac{k_3}{2} \coth(X_i(t)) + k_2 \coth(2X_i(t)) \\
        & + \frac{k_1}{2} \sum_{j \neq i} \big( \coth(X_i(t) - X_j(t)) + \coth(X_i(t) + X_j(t)) \big) \Bigg) dt,
    \end{split}
    \end{equation}
    and $X(t)\in\mathcal{C}$ for  any $t>0$.
\end{lemma}

\begin{proof}
    Let  $X(0)\in\bar{\mathcal{C}}$.
    We rely on \cite[Proposition~4.1]{Schapira_Heckman} to assert the strong existence and path-wise uniqueness of a solution to 
    \[
        dX_i(t)
        =
        d\beta_i(t)
        +
        \left(
        \frac{k_3+k_2}{2} \coth(X_i(t)) 
        +
        \frac{k_1}{2}
        \sum_{j\neq i}
        \big(
            \coth(X_i(t)-X_j(t))
            +
            \coth(X_i(t)+X_j(t))
        \big)
        \right)
        dt
        ,
    \]
    and by the discussion above Proposition~4.2 in \cite{Schapira_Heckman}, $X(t)\in\mathcal{C}$ for any $t>0$ almost surely. 
     We define an exponential local martingale similar to the one used in the proof of Theorem  \ref{characteristic function area}
    \[
        D(t) = 
        \exp\left(
            \frac{k_2}{2}\sum_{i=1}^r\int_0^t \tanh(X_i(s))d\beta_i(s)
            -
            \frac{k_2^2}{8}\sum_{i=1}^r\int_0^t \tanh^2(X_i(s))ds
        \right)
        .
    \]
    For the same reason as in the proof we just mentioned, it is a true martingale.
    Hence we can define the probability measure 
    \[
        \mathbb{Q}_{\mid\mathcal{F}_{t}}=D(t)\mathbb{P}_{\mid \mathcal{F}_{t}}
        ,
    \]
    under which $X(t),t\ge0$ solves the stochastic differential equation
    \begin{equation*}
    \begin{split}
        dX_i(t) = d\tilde\beta_i(t) &+ \Bigg( \frac{k_3}{2} \coth(X_i(t)) + k_2 \coth(2X_i(t)) \\
        & + \frac{k_1}{2} \sum_{j \neq i} \big( \coth(X_i(t) - X_j(t)) + \coth(X_i(t) + X_j(t)) \big) \Bigg) dt,
    \end{split}
    \end{equation*}
    where $\tilde\beta_i(t)= \beta_i(t) - \int_{0}^{t}2^{-1}k_2 \tanh(X_i(s))ds$, $i=1,\dots,r$ is a standard Brownian motion. That is, we have proved the weak existence of a solution to our stochastic differential equation and that almost surely this solution lies in $\mathcal{C}$ for any $t>0$ and satisfies $X_1>\dots>X_r>0$.

    We now prove the path-wise uniqueness. The proof relies on the $1$-Lipschitzness of $\tanh$, the non-increasing property of the function $\coth$ used in \cite{Schapira_Heckman} to prove a similar result when multiplicities are all non-negative and Gronwall's lemma.
    Let $X^1,X^2$ be two solutions driven by the same Brownian motion $\beta$ and such that $X^1(0)=X^2(0)$ almost surely.
    By a Girsanov change of measure argument similar to the one used above to derive weak existence, we get $X^1_1>\dots>X^1_r>0$ and $X^2_1>\dots>X^2_r>0$ almost surely.
    Let $\Delta(t)=(X^1-X^2)(t)$ and $V=|\Delta(t)|^2=\sum_{i=1}^r\Delta_i^2(t)$. Notice $\frac{d}{dt}V(t)=2 \sum_{i=1}^r \Delta_i(t) \frac{d}{dt}\Delta_i(t)$ where
    \[
        \Delta_i(t)
        =
        \int_{0}^t
            \big(
                b_i(X^1(s)) - b_i(X^2(s))
            \big)
            +
            \frac{k_2}{2}
            \big(
                \tanh(X_i^1(s)) - \tanh(X_i^2(s))
            \big)
        ds
        ,
    \]
    for $b_i(x) = \frac{k_2 + k_3}{2}\coth(x_i)  + \frac{k_1}{2}\sum_{j\neq i}\bigl(\coth(x_i-x_j)+\coth(x_i+x_j)\bigr)$. Since $\tanh$ is $1$-Lipschitz, we can write
    \[
        \Big| 
            \sum_{i=1}^r
                \Delta_i
                \frac{k_2}{2}
                \big(
                    \tanh(X_i^1(s)) - \tanh(X_i^2(s))
                \big)
        \Big|
        \le
        \frac{|k_2|}{2}
        \sum_{i=1}^r
            \Delta_i^2(s)
        =
        \frac{|k_2|}{2} V(s)
    \]
    and using that $\coth$ is non-increasing on $\mathbb{R}_{>0}$, we also have
    \[
    \sum_{i=1}^r \Delta_i \bigl( b_i(X^1(s)) - b_i(X^2(s)) \bigr) \le 0.
    \]
    Indeed, the non-interacting terms give $(x_i-y_i)(\coth x_i-\coth y_i)\le 0$ because $\coth$ is non-increasing on the set $(0,\infty)$ to which belongs $X^1_i,X_i^2$ almost surely.
    For the interaction term we symmetrize:
    \begin{align*}
    \sum_{i=1}^r \Delta_i \sum_{j\neq i} \coth(X_i^1-X_j^1)
    &= \sum_{i<j} (\Delta_i-\Delta_j)\coth(X_i^1-X_j^1), \\
    \sum_{i=1}^r \Delta_i \sum_{j\neq i} \coth(X_i^1+X_j^1)
    &= \sum_{i<j} (\Delta_i+\Delta_j)\coth(X_i^1+X_j^1),
    \end{align*}
    and the same identities hold for $X^2$.
    The differences then become
    \[
    \frac{k_1}{2}\sum_{i<j}(\Delta_i-\Delta_j)\bigl(\coth(X_i^1-X_j^1)-\coth(X_i^2-X_j^2)\bigr)
    +\frac{k_1}{2}\sum_{i<j}(\Delta_i+\Delta_j)\bigl(\coth(X_i^1+X_j^1)-\coth(X_i^2+X_j^2)\bigr).
    \]
    Since $X^1_1>\dots>X^1_r>0$ and $X^2_1>\dots>X^2_r>0$ almost surely, the differences $X_i-X_j$ ($i<j$) are strictly positive and the sums $X_i+X_j$ are strictly positive.
    Writing $\Delta_i-\Delta_j = (X_i^1-X_j^1)-(X_i^2-X_j^2)$ and $\Delta_i+\Delta_j = (X_i^1+X_j^1)-(X_i^2+X_j^2)$, each term is of the form $(u-v)(\coth u-\coth v)\le 0$ with $u,v>0$, again because $\coth$ is non-increasing.
    Hence the whole expression is non-positive.
    Therefore since $\Delta(0)=0$ we get
    \[
        V(t)\le 
        |k_2|
        \int_{0}^t
            V(s)ds
    \]
    and the pathwise uniqueness follows from a direct application of Gronwall's lemma.
    
    Finally, using the Yamada--Watanabe theorem \cite{YamadaWatanabe, Jacod1980} we get strong existence and uniqueness of a solution to the stochastic differential equation we are interested in.
\end{proof}

Consider the measure $m_{k}$ defined by its density with respect to the Lebesgue measure as follows:
\begin{align}
dm_{k}(\tau)
&=\;
 \prod_{i=1}^r |\sinh \tau_i|^{k_3}\,|\sinh (2\tau_i)|^{k_2}\;
 \prod_{1\le i<j\le r}\Bigl|\sinh(\tau_i-\tau_j)\,\sinh(\tau_i+\tau_j)\Bigr|^{k_1} d\tau
\label{eq:mu-main-bis}.
\end{align}
Under the assumptions made in the lemma above on the multiplicities $k_i$, $m_{k}$ is a Radon measure on the Borel sets of $\bar{\mathcal{C}}$ that is symmetrizing for the elliptic operator $\frac{1}{2}\mathcal{L}_{k}$.
More precisely, let
\[
\mathcal D:=\{f|_{\mathcal C}:f\in C_c^\infty(\mathbb R^r)^W\}.
\]
The closure of
\[
\mathcal E_k(f,g)
=\frac12\int_{\mathcal C}\langle\nabla f,\nabla g\rangle\,dm_k,
\qquad f,g\in\mathcal D,
\]
is a regular strongly local Dirichlet form. Its Hunt process solves the martingale problem for $\frac12\mathcal L_k$ and hence, by Lemma~\ref{sol SDE}, agrees with the pathwise unique diffusion constructed above. We denote 
its heat kernel with respect to $m_k$ by $p_t^k$; see, for example, Proposition~4.11 and Theorem~4.23 in \cite{baudoin2014diffusion}.


\subsection{Exponential integrability of the Heckman-Opdam Processes}

From now on we further assume $k_2\ge0$. This allows us to derive a usual result that will be applied multiple times in the rest of this appendix as well as throughout the main text to the processes $\tau_i(t)=\arctanh \left( \lambda_i(t) \right)$ where $\lambda_i(t)$, $1 \le i \le r$, are the spectral values of a Brownian motion $(X(t))_{t \ge 0}$ on $\Omega$.

\begin{lemma}
\label{exponential integrability}
    Let $k_1\ge1,k_2\ge0$ and $k_3 \in \mathbb{R}$ be such that $k_2+k_3\ge 1$.
    Consider a stochastic  process $(X_i(t),1\le i \le r)$ on $\bar{\mathcal{C}}$ with generator $\frac{1}{2}\mathcal{L}_{k_1,k_2,k_3}$.
    Then for every $t \ge 0$  and $\alpha  \ge 0$
    \[
        \mathbb{E} \left( e^{\alpha X_i(t)} \right) 
        <+\infty
        , 
        \qquad i=1,\cdots,r.
    \]
\end{lemma}
\begin{proof}
Let \(X(0)=x\in\overline{\mathcal C}\) be fixed. The function
$V(x)=|x|^2$ used below is smooth across the walls, and the expressions
$u\coth u$ arising from its drift extend continuously to $u=0$.
Thus the following It\^o argument also applies when $x$ lies on a wall,
since $X(t)\in\mathcal C$ for every $t>0$ almost surely.

Write
\[
\mathcal L=\frac12 \mathcal L_{k_1,k_2,k_3}
  =\frac12\sum_{i=1}^r \frac{\partial^2}{\partial x_i^2}
   +\sum_{i=1}^r \mu_i(x)\frac{\partial}{\partial x_i},
\]
where
\[
\mu_i(x)
=
\frac{k_3}{2}\coth x_i
+k_2\coth(2x_i)
+\frac{k_1}{2}\sum_{j\ne i}
\bigl(\coth(x_i-x_j)+\coth(x_i+x_j)\bigr).
\]
Let
\[
V(x)=|x|^2=\sum_{i=1}^r x_i^2 .
\]
Then
\[
 \mathcal LV(x)=r+2x\cdot \mu(x).
\]
Using
\[
2\coth(2u)=\coth u+\tanh u,
\]
the one-particle contribution is
\[
k_3 x_i\coth x_i+2k_2 x_i\coth(2x_i)
=
(k_2+k_3)x_i\coth x_i+k_2x_i\tanh x_i.
\]
Since \(k_2\ge0\) and \(k_2+k_3\ge1\), and since
\[
u\coth u\le 1+u^2,\qquad u\tanh u\le u^2
\qquad (u>0),
\]
this contribution is bounded above by a constant plus a constant multiple
of \(x_i^2\).

For the interaction terms, grouping the pair \((i,j)\) with \((j,i)\) gives
\[
x_i\coth(x_i-x_j)+x_j\coth(x_j-x_i)
=
(x_i-x_j)\coth(x_i-x_j),
\]
and
\[
x_i\coth(x_i+x_j)+x_j\coth(x_j+x_i)
=
(x_i+x_j)\coth(x_i+x_j).
\]
Both are bounded above by a constant plus a quadratic function of \(x_i\)
and \(x_j\). Therefore there exist constants \(C_0,C_1>0\) such that
\[
\mathcal LV(x)\le C_0+C_1V(x),\qquad x\in \mathcal C.
\]

Let \(\eta\) be the positive solution on \([0,\infty)\) of
\[
\eta'(t)+C_1\eta(t)+2\eta(t)^2=0,\qquad \eta(0)=1,
\]
namely
\[
\eta(t)=\frac{C_1e^{-C_1t}}{C_1+2(1-e^{-C_1t})}.
\]
Define
\[
F(t,x)=\exp(\eta(t)V(x)).
\]
Then
\[
\nabla F=2\eta xF,\qquad
\mathcal LF=F\bigl(\eta \mathcal LV+2\eta^2V\bigr).
\]
Consequently,
\[
(\partial_t+\mathcal L)F
=
F\bigl(\eta'V+\eta \mathcal LV+2\eta^2V\bigr)
\le C_0\eta(t)F.
\]

Let
\[
\sigma_n=\inf\{s\ge0:\ |X(s)|\ge n\}.
\]
Applying It\^o's formula to \(F(s\wedge\sigma_n,X(s\wedge\sigma_n))\) and
taking expectations gives
\[
\mathbb E \left( F(t\wedge\sigma_n,X(t\wedge\sigma_n))\right)
\le
F(0,x)+C_0\int_0^t
\eta(s)\,
\mathbb E \left(  F(s\wedge\sigma_n,X(s\wedge\sigma_n))\right)\,ds .
\]
By Gronwall's lemma,
\[
\mathbb E  \left ( F(t\wedge\sigma_n,X(t\wedge\sigma_n)) \right)
\le
e^{|x|^2}
\exp\left(C_0\int_0^t\eta(s)\,ds\right).
\]
Letting \(n\to\infty\) and using Fatou's lemma yields
\[
\mathbb E \left( \exp\{\eta(t)|X(t)|^2\} \right)
\le
e^{|x|^2}
\exp\left(C_0\int_0^t\eta(s)\,ds\right)
<\infty .
\]
Since \(\eta(t)>0\), for every \(\alpha\ge0\),
\[
\alpha |X(t)|
\le
\eta(t)|X(t)|^2+\frac{\alpha^2}{4\eta(t)}.
\]
Hence
\[
\mathbb E \left( e^{\alpha |X(t)|}\right) <\infty.
\]
\end{proof}

As an immediate consequence, under the assumptions of the lemma above the processes $X_i(t)$, $t\ge0$, are integrable. Since $X_1(t)$ has non-negative drift, it is a submartingale, and so is $e^{\alpha X_1(t)}$.
For $p>1$, Doob's $L^p$ maximal inequality, the estimate $X_i\le X_1$, and the lemma applied with $p\alpha$ give
\[
    \mathbb{E}\!\left[\sup_{0\le s\le t}e^{\alpha X_i(s)}\right]
    \le \frac p{p-1}
    \left(\mathbb E e^{p\alpha X_1(t)}\right)^{1/p}
    <\infty
    ,
    \qquad
    \alpha>0
    .
\]

\subsection{Integral representation of the heat kernel}

When all multiplicities are positive, it is known \cite{Schapira_Heckman,Schapira_sharp_estimates} building on \cite{rosler1998} that the heat kernel $p_t^k$ admits a representation as an integral involving Heckman-Opdam hypergeometric functions $F_\lambda \left(k;\bullet\right)$ as a consequence of the inversion formula for the hypergeometric Fourier transform developed in \cite{Opdam1995}. 
When some multiplicities can be negative, one needs to deal with the poles of the integrand in the inversion formula; such an inversion formula has been derived in Theorem~5.3 of \cite{honda2024inversion}.
We first introduce the notations of the latter theorem, then relying on it we prove an integral representation of $p_t^k$ in the setting where $k_1\ge1,k_2\ge0$ and $k_3\in\mathbb{R}$ with the constraint $k_2+k_3\ge1$.

For $p\in\mathbb C[\zeta_1,\ldots,\zeta_r]^W$, let $D_p(k)$ denote
the commuting Heckman--Opdam differential operators, normalized by the
Harish--Chandra isomorphism $D_p(k)\mapsto p$; see
\cite[Equations~(1.8)--(1.9)]{honda2024inversion}.
The Heckman--Opdam hypergeometric function $F_\lambda(k;\bullet)$ is the
unique $W$-invariant real analytic function on $\mathbb R^r$, normalized by
$F_\lambda(k;0)=1$, that satisfies the joint eigenvalue system
\[
    D_p(k)F_\lambda(k;\bullet)
    =p(\lambda)F_\lambda(k;\bullet),
    \qquad p\in\mathbb C[\zeta_1,\ldots,\zeta_r]^W.
\]
For $p_2(\zeta)=\sum_{i=1}^r\zeta_i^2$, the normalization gives
$D_{p_2}(k)=\mathcal L_k+|\rho(k)|^2$. In particular,
\begin{align}
    \mathcal{L}_{k}\,F_\lambda \left(k;\tau\right)
    &=\bigl(
    \langle \lambda, \lambda\rangle
    -|\rho(k)|^2
    \bigr)\,
    F_\lambda \left({k};\tau\right)
\label{eq:eig}
\end{align}
where the operator $\mathcal{L}_{k}$ is extended $W$-invariantly from
$\mathcal C$ to the complement of the reflection hyperplanes in $\mathbb{R}^r$,
and $\rho(k)$ is defined by
\[
    \rho(k)
    =
    \frac{1}{2}
    \sum_{\alpha\in\mathcal{R}_+} k_\alpha \alpha
    .
\]

We use the same symbol $m_k$ for the $W$-invariant measure on
$\mathbb R^r$ defined by the density in \eqref{eq:mu-main-bis}. Let
$\mathrm{C}_{c}^\infty(\mathbb{R}^r)^W$ denote the family of $W$-invariant
compactly supported smooth functions on $\mathbb{R}^r$. It is dense in
$L^2(\mathbb{R}^r,|W|^{-1}dm_k)^W$, and restriction to $\mathcal C$ gives
the natural isometry
\[
L^2(\mathbb{R}^r,|W|^{-1}dm_k)^W
\simeq L^2(\mathcal C,dm_k).
\]
For $k\in\mathbb{C}^3$, the hypergeometric Fourier transform $\mathcal{F}_{k}$ of $f \in \mathrm{C}_{c}^\infty(\mathbb{R}^r)^W$ is defined by
\[
    \mathcal{F}_{k} f(\lambda) 
    = 
    \int_{\mathcal{C}} 
        f(x) 
        F_{\lambda}(k; x) 
        dm_k(x)
    = 
    \frac{1}{|W|} 
    \int_{\mathbb{R}^r} 
        f(x) F_{\lambda}(k; x)
        dm_k(x)
    .
\]

Let $\alpha_1,\dots,\alpha_r$ denote the simple positive roots defined by
\[
    \alpha_i=2e_i-2e_{i+1}
    \quad(1 \le i \le r-1),
    \qquad
    \alpha_r=2e_r
\]
and write $\mathcal{B}=\{\alpha_1,\dots,\alpha_r\}$ for the set of simple roots in $\mathcal{R}_+$.
We define for $0\le j \le r$ subsets of $\mathcal{B}$ by
\[
    \Theta_j 
    = 
    \{ \alpha_l \mid r - j + 1 \le l \le r \}.
\]
and write $\langle \Theta_j \rangle$ for the subset of $\mathcal{R}$ made up of linear combinations of elements of $\Theta_j$, and $\langle \Theta_j \rangle_+ = \langle \Theta_j\rangle \cap \mathcal{R}_+$.

We now introduce the $c$-functions used to write the density of the symmetric Plancherel measures that appear in the formula of the heat kernel. Let $\alpha^\vee$ denote the coroot $2\alpha/\langle \alpha,\alpha \rangle$ of $\alpha$, and adopt the standard convention $k_\beta=0$ when $\beta\notin\mathcal R$. Define
\[
    \tilde c_{k,\alpha}(\lambda) 
    = 
    \frac{
        \Gamma\left(\langle\lambda, \alpha^\vee\rangle + \frac{1}{2}k_{\frac{1}{2}\alpha}\right)
    }{
        \Gamma\left(\langle\lambda, \alpha^\vee\rangle + \frac{1}{2}k_{\frac{1}{2}\alpha} + k_\alpha\right)
    }
\]
where $\Gamma$ is the usual gamma function and
\begin{align*}
    &\prescript{}{j}{\tilde{c}}_k(\lambda)
    =
    \prod_{\alpha \in \langle \Theta_j \rangle_+}
    \tilde c_{k,\alpha}(\lambda)
    ,\qquad
    &&\prescript{j}{}{\tilde{c}}_k(\lambda)
    =
    \prod_{\alpha \in \mathcal{R}_+ \setminus \langle \Theta_j \rangle_+}
    \tilde c_{k,\alpha}(\lambda)
    \\
    &\prescript{}{j}c_k(\lambda)
    =
    \frac{\prescript{}{j}{\tilde{c}}_k(\lambda)}{\prescript{}{j}{\tilde{c}}_k(\rho(k))}
    ,
    \qquad
    &&\prescript{j}{}c_{k}(\lambda)
    =
    \frac{\prescript{j}{}{\tilde{c}}_k(\lambda)}{\prescript{j}{}{\tilde{c}}_k(\rho(k))},
\end{align*}
so that the Harish-Chandra $c$-function can be written
\[
    c_k(\lambda) = \prescript{}{j}c_k(\lambda) \prescript{j}{}c_k(\lambda)
    .
\]
Let $\boldsymbol{\alpha}=(k_2+k_3-1)/2$ and $\boldsymbol{\beta}=(k_2-1)/2$.
For $1\le j \le r$ we introduce the finite subsets of $\mathbb{R}^j$ defined by
\begin{align*}
    D_{k,j}
    = 
    \{ & (\lambda_1, \dots, \lambda_j) \in \mathbb{R}^j; \lambda_1 + |\boldsymbol{\beta}| - \boldsymbol{\alpha} - 1 \in 2\mathbb{N},
    \lambda_j < 0, \\
    & \lambda_{l+1} - \lambda_l - k_1 \in 2\mathbb{N} \ (1 \le l \le j-1) \}.
\end{align*}
which are non-empty if and only if $\boldsymbol{\alpha} - |\boldsymbol{\beta}| + (j-1)k_1 + 1<0$, allowing one to recover the formulas from \cite{Opdam1995} when all multiplicities are non-negative. For convenience, when $j=0$ we write $D_{k,0} = \{0\}$.

Let $a_j=\operatorname{span}(e_{r-j+1},\ldots,e_r)$ and $a^j=a_j^\perp$. We identify $(\lambda_1,\ldots,\lambda_j)\in\mathbb R^j$ with $\sum_{\ell=1}^j\lambda_\ell e_{r+1-\ell}\in a_j$. Thus any $\lambda\in\mathbb{R}^r$ can be uniquely decomposed as $\lambda=\lambda_{a_j} + \lambda_{a^j}$ with $\lambda_{a_j}\in a_j$ and $\lambda_{a^j}\in a^j$.
For $0\le j\le r$, let
\[
    W_{\Theta_j}:=\langle r_\alpha:\alpha\in\Theta_j\rangle
\]
be the parabolic Weyl subgroup generated by the reflections in the roots of
$\Theta_j$. For $\lambda_{a_j}\in a_j$, write
\[
    W_{\Theta_j}^{\lambda_{a_j}}
    :=\{w\in W_{\Theta_j}:w\lambda_{a_j}=\lambda_{a_j}\}
\]
for its stabilizer in $W_{\Theta_j}$.
For $j = 0, \dots, r$ and $\lambda_{a_j} = (\lambda_1, \dots, \lambda_j) \in D_{k,j}$ we define a positive number $d_j(\lambda_{a_j}, k)$ as follows. If $k_1>0$ and $r\neq1$ we define
\begin{align*}
    d_{j}(\lambda_{a_j}, k) 
    &= \prescript{}{j}{\tilde{c}}_k(\rho(k))^2 \nonumber \\
    &\times \prod_{l=1}^j \frac{-2^{2\boldsymbol{\alpha}-2\boldsymbol{\beta}-1} \lambda_l}{\pi} \frac{\Gamma\left(\frac{1}{2}(\lambda_l + \boldsymbol{\alpha} + |\boldsymbol{\beta}| + 1)\right) \Gamma\left(\frac{1}{2}(-\lambda_l + \boldsymbol{\alpha} + |\boldsymbol{\beta}| + 1)\right)}{\Gamma\left(\frac{1}{2}(\lambda_l - \boldsymbol{\alpha} + |\boldsymbol{\beta}| + 1)\right) \Gamma\left(\frac{1}{2}(-\lambda_l - \boldsymbol{\alpha} + |\boldsymbol{\beta}| + 1)\right)} \nonumber \\
    &\times \prod_{1 \leq q < p \leq j} \frac{(\lambda_q^2 - \lambda_p^2) \Gamma\left(\frac{1}{2}(\lambda_p - \lambda_q + k_1)\right) \Gamma\left(\frac{1}{2}(-\lambda_q - \lambda_p + k_1)\right)}{4 \Gamma\left(\frac{1}{2}(\lambda_p - \lambda_q - k_1 + 2)\right) \Gamma\left(\frac{1}{2}(-\lambda_q - \lambda_p - k_1 + 2)\right)}.
\end{align*}
If $k_1=0$ or $r = 1$, we write
\begin{align*}
    d_{j}(\lambda_{a_j}, k) 
    &= \prescript{}{j}{\tilde{c}}_k(\rho(k))^2 
    \left| W_{\Theta_j}^{\lambda_{a_j}} \right|^{-1} \nonumber \\
    &\times \prod_{l=1}^j \frac{-2^{2\boldsymbol{\alpha}-2\boldsymbol{\beta}-1} \lambda_l}{\pi} \frac{\Gamma\left(\frac{1}{2}(\lambda_l + \boldsymbol{\alpha} + |\boldsymbol{\beta}| + 1)\right) \Gamma\left(\frac{1}{2}(-\lambda_l + \boldsymbol{\alpha} + |\boldsymbol{\beta}| + 1)\right)}{\Gamma\left(\frac{1}{2}(\lambda_l - \boldsymbol{\alpha} + |\boldsymbol{\beta}| + 1)\right) \Gamma\left(\frac{1}{2}(-\lambda_l - \boldsymbol{\alpha} + |\boldsymbol{\beta}| + 1)\right)}
    .
\end{align*}

\begin{theorem}\label{heat kernel formula}
    Let $k_1\ge 1$, $k_2\ge0$, and $k_3 \in \mathbb{R}$ be such that $k_2+k_3\ge 1$. Then, for every $t>0$ and $\tau,\eta\in\bar{\mathcal{C}}$, we can write
    \[
        p_{t}^{k}(\tau, \eta)
        =
        p_{t}^{k, \text{cont}}(\tau, \eta)
        +
        \sum_{j=1}^{r} p_{t}^{k,  j}(\tau, \eta)
    \]
    where
    \[
        p_{t}^{k, \text{cont}}(\tau, \eta) 
        = 
        \int_{i \mathbb{R}^{r}} e^{-\frac{t}{2}\left(|\lambda|^{2}+|\rho(k)|^{2}\right)} F_{\lambda}(k ; \tau) F_{\lambda}(k ;-\eta) 
        \frac{C(k)\, d \lambda}{\left|c_{k}(\lambda)\right|^{2}}
    \]
    and 
    \[
        p_{t}^{k, j}(\tau, \eta)
        =
        \sum_{\xi \in D_{k,j}} 
            d_{j}(\xi, k) 
            \int_{i a^j}
                e^{
                    \frac{t}{2}\left(\langle\xi+\lambda, \xi
                    +
                    \lambda \rangle
                    -
                    |\rho(k)|^{2}\right)
                } 
                F_{\xi+\lambda}( k ; \tau) 
                F_{\xi+\lambda}(k ;-\eta) 
                \frac{C_j(k)\,d\lambda}{\left|\prescript{j}{}c_k(\xi+\lambda)\right|^{2}}
    \]
    for constants $C(k),C_1(k),\dots,C_r(k)\in\mathbb{C}$ such that each integral is real-valued. More precisely, the weight $\delta_k$ used in \cite[Equations~(5.8)--(5.10) and Theorem~5.3]{honda2024inversion} satisfies
    \[
        \delta_k(\tau)\,d\tau=2^{\kappa(k)}dm_k(\tau),
        \qquad \kappa(k)=r(k_2+k_3)+r(r-1)k_1,
    \]
    so the constants above equal $2^{\kappa(k)}$ times the corresponding contour constants in that reference. In particular, when $k_3\ge-2$, the sum term in $p_t^k$ vanishes.
\end{theorem}

\begin{proof}

    Let $q_t^k(\tau,\eta)$ denote the candidate kernel displayed in the
    statement. For convenience, we suppress the displayed $k$-dependence of
    the contour constants and write $C_0=C$. The convention $D_{k,0}=\{0\}$
    makes the continuous part the $j=0$ term.

    For the particular case, it suffices to notice that the condition $\boldsymbol{\alpha} - |\boldsymbol{\beta}| + (j-1)k_1 + 1<0$ is satisfied for no $1 \le j \le r$ when $k_3\ge -2$.
    
    \textbf{Step 1: Convergence and growth bounds.}
    As an immediate consequence of Lemma~2.7 in \cite{honda2024inversion}, for $\nu$ a multi-index and $K$ a compact subset of $\mathbb{R}^r$, there exists $C_K >0$ and $N\in\mathbb{N}$ such that $\forall x\in K$ and $\forall\mu,\lambda \in \mathbb{R}^r$,
    \[
        |\partial_{x^\nu} F_{\mu + i \lambda}\left(k;x\right)| \leq C_K (1 + |\mu + i \lambda|^{N})
        e^{
            \max_{w\in W}
                \langle \mu, w^{-1} x \rangle
        }
        .
    \]
    In particular, when $\mu$ is fixed, the function
    $\lambda\mapsto|\partial_{x^\nu} F_{\mu + i \lambda}(k;x)|$
    has at most polynomial growth, uniformly for $x$ in a fixed compact set.
    Moreover, by Lemma~2.4 in \cite{NarayananPasquale2022}, there exists a constant $\kappa$ such that $\forall\mu,\lambda \in \mathbb{R}^r$,
    \[
        |F_{\mu + i \lambda}\left(k;x\right)| 
        \leq
        \kappa 
        e^{
            \max_{w\in W}
                \langle \mu, w^{-1} x \rangle
        }
        ,
        \qquad
        \forall x\in \mathbb{R}^r
        .
    \]
    A direct application of Lemma~1.2 in \cite{honda2024inversion} provides the at-most polynomial growth of the $c$-functions, using Lemma~5.2 when $j\neq0$ to ensure that assumptions of Lemma~1.2 are satisfied.
    Let $\mu \in D_{k,j}$ for some $0 \le j \le r$, then there exists $N>0$ and $C_j>0$ such that for every $\lambda\in a^j$,
    \[
        |\prescript{j}{}c_k(\mu + i\lambda)|^{-2}
        \le
        C_j (1+ |\lambda|)^N
        .
    \]
    Since the $D_{k,j}$, $0\le j\le r$ are finite sets, all the quantities considered in the statement are well defined for $t>0$ and $\tau,\eta\in\mathbb{R}^r=\cup_{w\in W}w\cdot\bar{\mathcal{C}}$. Moreover the candidate heat kernel is continuous on $(0,\infty)\times\bar{\mathcal{C}}\times\bar{\mathcal{C}}$.
    
    \textbf{Step 2: Construction of a classical heat solution.}
    Set
    \[
        \Sigma_k
        :=\bigsqcup_{j=0}^r\ \bigsqcup_{\xi\in D_{k,j}}
        (\xi+i a^j),
        \qquad
        \varepsilon_k(\zeta)
        :=\frac12\bigl(\langle\zeta,\zeta\rangle-|\rho(k)|^2\bigr).
    \]
    Let $\pi_k$ denote the positive inversion measure on these components,
    including the constants $C_j(k)$ and the weights $d_j(\xi,k)$ from the
    statement, with the convention $d_0(0,k)=1$. The contour integrals are
    understood through the real parametrization $\zeta=\xi+i\lambda$,
    $\lambda\in a^j$. With this normalization, the inversion formula of
    \cite[Theorem~5.3]{honda2024inversion} uses the transform
    $\mathcal F_k$ defined above.

    Fix $f\in\mathcal D$, extended $W$-invariantly to $\mathbb R^r$, and
    write $\widehat f=\mathcal F_kf$. By
    \cite[Proposition~2.1]{honda2024inversion}, for every $N\ge0$ there is
    a constant $C_{f,N}$ such that
    \[
        |\widehat f(\xi+i\lambda)|
        \le C_{f,N}(1+|\lambda|)^{-N},
        \qquad
        \xi\in D_{k,j},\quad \lambda\in a^j,\quad 0\le j\le r.
    \]
    Indeed, the possible real parts $\xi$ form a finite set. Define
    \[
        u(t,x)
        :=\int_{\Sigma_k}
        e^{t\varepsilon_k(\zeta)}
        F_\zeta(k;x)\widehat f(\zeta)\,d\pi_k(\zeta),
        \qquad t\ge0,\quad x\in\mathbb R^r.
    \]
    On the component $\xi+i a^j$,
    \[
        \varepsilon_k(\xi+i\lambda)
        =\frac12\bigl(|\xi|^2-|\lambda|^2-|\rho(k)|^2\bigr).
    \]
    Thus the rapid decay of $\widehat f$, together with the polynomial
    bounds in Step~1, permits differentiation under the integral in time
    and space, locally uniformly for $t\ge0$ and $x\in\mathbb R^r$.
    In particular, $u$ is $W$-invariant, smooth in $x$, and continuously
    differentiable in $t$ up to $t=0$. The inversion formula and
    \eqref{eq:eig} give
    \[
        u(0,x)=f(x),
        \qquad
        \partial_tu(t,x)=\frac12\mathcal L_ku(t,x),
        \quad t>0,\quad x\in\mathcal C.
    \]
    Moreover, the global bound for $F_\zeta$ in Step~1 implies that, for
    every $T>0$, there is a constant $C_{f,T}$ such that
    \[
        \sup_{0\le t\le T}|u(t,x)|
        \le C_{f,T}e^{M|x|},
        \qquad x\in\mathbb R^r,
        \qquad
        M:=1+\max_{0\le j\le r}\max_{\xi\in D_{k,j}}|\xi|,
    \]
    where empty sets are omitted from the maximum.
    Since $-\mathrm{Id}\in W$, we have
    $F_\zeta(k;-\eta)=F_\zeta(k;\eta)$. Fubini's theorem therefore gives
    \[
        u(t,x)=\int_{\bar{\mathcal C}}
        q_t^k(x,\eta)f(\eta)\,dm_k(\eta),
        \qquad t>0.
    \]

    \textbf{Step 3: Identification with the diffusion semigroup.}
    We use the stopped backward heat-equation argument from
    \cite[proof of Theorem~4.2.1]{Hsu2002}, with exponential domination
    in place of boundedness. Let $X$ be the diffusion from
    Lemma~\ref{sol SDE}, started at any $x\in\bar{\mathcal C}$, and fix
    $T>0$. For $N>|x|$, set
    \[
        \sigma_N:=\inf\{s\ge0:|X(s)|\ge N\}.
    \]
    The process $X$ is a semimartingale and $X(s)\in\mathcal C$ for every
    $s>0$ almost surely. Since $u$ extends smoothly across the walls,
    It\^o's formula applies to
    $u(T-s\wedge\sigma_N,X(s\wedge\sigma_N))$, even if $x$ is on a wall.
    The backward heat equation cancels its drift. Its stochastic
    integrand is bounded on $[0,T]\times\{|y|\le N\}$, so the stopped
    process is a true martingale. Consequently,
    \[
        u(T,x)
        =\mathbb E_x\!\left[
        u(T-T\wedge\sigma_N,X(T\wedge\sigma_N))\right].
    \]
    The global solution satisfies $\sigma_N\to\infty$ almost surely.
    The integrands above are bounded in absolute value by
    \[
        C_{f,T}\exp\!\left(M\sup_{0\le s\le T}|X(s)|\right),
    \]
    which is integrable by the maximal estimate following
    Lemma~\ref{exponential integrability}, since
    $|X(s)|\le\sqrt r\,X_1(s)$. Dominated convergence now yields
    \[
        u(T,x)=\mathbb E_x[f(X(T))]
        =\int_{\bar{\mathcal C}}
        q_T^k(x,\eta)f(\eta)\,dm_k(\eta).
    \]
    As $\mathcal D$ determines Radon measures on $\bar{\mathcal C}$,
    this identifies $q_T^k(x,\eta)\,dm_k(\eta)$ with the transition
    measure of $X$ for every $x\in\bar{\mathcal C}$. The joint continuity
    proved in Step~1 supplies the asserted heat-kernel version, which we
    denote by $p_t^k$.
\end{proof}


\section{Irreducible Cartan domains}
\label{sec:classical-domains-bergman}

This appendix records the explicit formulas for the irreducible bounded
symmetric domains used in the paper. For each Cartan domain we give the
bounded realization, the associated Hermitian positive Jordan triple system,
the generic norm, the genus, the Bergman operator, the canonical Hermitian
trace form, the K\"ahler potential, and the corresponding K\"ahler
\(d^c\)-potential form. In the tube-type cases we also record the
\(d^c\)-log-determinant form.

Throughout, \(V\) denotes the underlying Hermitian positive Jordan triple
system. The Bergman operator is
\begin{equation}
\label{eq:bergman-operator-def-appendix}
B(z,w)=\mathrm{Id}-D(z,w)+Q(z)Q(w),
\end{equation}
where
\[
D(z,w)u=\{z,w,u\},
\qquad
Q(z)u=\frac{1}{2}\{z,u,z\}.
\]
In all cases the normalized invariant metric is
\[
h_z(u,v)=\frac1\gamma\bigl(B(z,z)^{-1}u\mid v\bigr),
\]
and a canonical \(K\)-invariant K\"ahler potential is
\begin{equation}
\label{eq:kahler-potential-appendix}
\Phi(z)=\frac1{2\gamma}\log K(z,z)=-\frac12\log N(z,z),
\end{equation}
where \(N\) is the generic norm and \(\gamma\) is the genus. The K\"ahler
\(d^c\)-potential form is
\begin{equation}
\label{eq:dc-potential-appendix}
\alpha_z
=
d^c
\Phi(z)
=
\frac{i}{4}\,
\frac{(\partial-\bar\partial)N(z,z)}{N(z,z)}.
\end{equation}
When the Hermitian positive Jordan triple system is of tube type, equivalently
\(b=0\) in the simultaneous Peirce decomposition introduced in
\cref{Preliminaries on bounded symmetric domains}, we also give the
\(d^c\)-log-determinant form
\begin{equation}
\label{eq:dc-log-det-appendix}
\theta_z
=
2 d^c \log \mathcal J(z)
=
2 d^c \log |\mathrm{det}_{ \mathcal J}|(z)
=
i\,
\frac{(\bar\partial-\partial)|\mathrm{det}_{ \mathcal J}|(z)}
     {|\mathrm{det}_{ \mathcal J}|(z)},
\qquad z\in\Omega^\star .
\end{equation}

The corresponding Hermitian symmetric spaces, representations, and tube-type conditions
are summarized as follows:
\[
\renewcommand{\arraystretch}{1.25}
\begin{array}{c|c|c|c}
\text{Type} & G/K & \gamma & \text{Tube type} \\
\hline
\mathrm{I}_{m,n}
&
SU(m,n)/S(U(m)\times U(n))
&
m+n
&
m=n
\\
\mathrm{II}_{n}
&
SO^\ast(2n)/U(n)
&
2(n-1)
&
n \text{ even}
\\
\mathrm{III}_{n}
&
Sp(n,\mathbb R)/U(n)
&
n+1
&
\text{yes}
\\
\mathrm{IV}_{n}
&
SO_0(n,2)/(SO(2) \times SO(n))
&
n
&
\text{yes}
\\
\mathrm{V}
&
E_{6}^{-14}/(\mathrm{Spin}(10)U(1))
&
12
&
\text{no}
\\
\mathrm{VI}
&
E_{7}^{-25}/(E_6U(1))
&
18
&
\text{yes}
\end{array}
\]
We refer to \cite[4.14 \& 4.17]{Loos1977} and
\cite[VI.5 in Part V]{Book_complexdomains} for the classification of simple
Hermitian positive Jordan triple systems and their associated bounded symmetric
domains.
\subsection*{Type I: rectangular matrices \texorpdfstring{\((1\le m\le n)\)}{}}

The type-I domain is
\[
\Omega_{\mathrm I}(m,n)
=
\{Z\in M_{m,n}(\mathbb C): I_m-ZZ^\ast>0\}.
\]
The triple product is
\[
\{X,Y,Z\}
=
XY^\ast Z+ZY^\ast X.
\]
The generic norm and genus are
\[
N(Z,Z)=\mathrm{det}(I_m-ZZ^\ast),
\qquad
\gamma=m+n.
\]
The Bergman operator is
\[
B(Z,Z)W
=
(I_m-ZZ^\ast)\,W\,(I_n-Z^\ast Z).
\]
The canonical Hermitian trace form is
\[
(X|Y)=(m+n)\tr(XY^\ast).
\]
The K\"ahler potential is
\[
\Phi(Z)=-\frac12\log\mathrm{det}(I_m-ZZ^\ast).
\]
The K\"ahler \(d^c\)-potential form is
\[
\alpha
=
\frac{1}{4i}
\tr\left(
(I_m-ZZ^\ast)^{-1}
(dZ\,Z^\ast-Z\,dZ^\ast)
\right).
\]

This Hermitian positive Jordan triple system is of tube type if and only if
\(m=n\). In the tube-type case,
\[
\mathcal J(z)=\sqrt{|\mathrm{det}(ZZ^\ast)|}
\]
and
\[
    \theta_Z
    =
    \im \tr\left(
        Z^{-1} dZ
    \right),
    \qquad
    Z\in\Omega^\star
    .
\]

\subsection*{Type II: complex skew-symmetric matrices \texorpdfstring{\((2\le n)\)}{}}

The type-II domain is
\[
\Omega_{\mathrm{II}}(n)
=
\{Z\in M_n(\mathbb C): Z^{\mathsf T}=-Z,\ I_n-ZZ^\ast>0\}.
\]
The triple product is the restriction of the type-I triple product:
\[
\{X,Y,Z\}
=
-\left(X\bar Y Z+Z\bar Y X\right).
\]
The Bergman operator is
\[
B(Z,Z)W
=
(I_n+Z\bar Z)\,W\,(I_n+\bar Z Z).
\]
The generic norm and genus are
\[
N(Z,Z)^2=\mathrm{det}(I_n+Z\bar Z),
\qquad
\gamma=2(n-1).
\]
The canonical Hermitian trace form is
\[
(X|Y)=(n-1)\tr(XY^\ast).
\]
The K\"ahler potential is
\[
\Phi(Z)=-\frac14\log\mathrm{det}(I_n+Z\bar Z).
\]
The K\"ahler \(d^c\)-potential form is
\[
\alpha
=
\frac{1}{8i}
\tr\left(
(I_n+Z\bar Z)^{-1}
(Z\,d\bar Z-dZ\,\bar Z)
\right).
\]

This Hermitian positive Jordan triple system is of tube type if and only if
\(n\) is even. In that case,
\[
\mathcal J(z)
=
\left(|\mathrm{det}(Z\bar Z)|\right)^{1/4}
\]
and
\[
    \theta_Z
    =
    \frac{1}{2} \im \tr\left(
        Z^{-1} dZ
    \right),
    \qquad
    Z\in\Omega^\star
    .
\]

\subsection*{Type III: complex symmetric matrices \texorpdfstring{\((1\le n)\)}{}}

The type-III domain is
\[
\Omega_{\mathrm{III}}(n)
=
\{Z\in M_n(\mathbb C): Z^{\mathsf T}=Z,\ I_n-Z\bar Z>0\}.
\]
The triple product is the restriction of the type-I triple product:
\[
\{X,Y,Z\}
=
X\bar Y Z+Z\bar Y X.
\]
The generic norm and genus are
\[
N(Z,Z)=\mathrm{det}(I_n-Z\bar Z),
\qquad
\gamma=n+1.
\]
The Bergman operator is
\[
B(Z,Z)W
=
(I_n-Z\bar Z)\,W\,(I_n-\bar Z Z).
\]
The canonical Hermitian trace form is
\[
(X|Y)=(n+1)\tr(X\bar Y).
\]
The K\"ahler potential is
\[
\Phi(Z)=-\frac12\log\mathrm{det}(I_n-Z\bar Z).
\]
The K\"ahler \(d^c\)-potential form is
\[
\alpha
=
\frac{1}{4i}
\tr\left(
(I_n-Z\bar Z)^{-1}
(dZ\,\bar Z-Z\,d\bar Z)
\right).
\]

This Hermitian positive Jordan triple system is of tube type. We have
\[
    \mathcal J(z)
    =
    \sqrt{|\det(Z\bar Z)|}
\]
and
\[
    \theta_Z
    =
    \im \tr\left(
        Z^{-1} dZ
    \right),
    \qquad
    Z\in\Omega^\star
    .
\]

\subsection*{Type IV: the Lie ball \texorpdfstring{\((n\ge 3)\)}{}}

The type-IV domain is
\[
\Omega_{\mathrm{IV}}(n)
=
\left\{
z\in\mathbb C^n:
0<1-2\langle z,\bar z\rangle+|\langle z,z\rangle|^2
\text{ and }
\langle z,\bar z\rangle<1
\right\},
\]
where
\[
\langle x,y\rangle=\sum_{j=1}^n x_jy_j.
\]
The triple product is
\[
\{x,y,z\}
=
2\left(
\langle x,\bar y\rangle z
+
\langle z,\bar y\rangle x
-
\langle x,z\rangle \bar y
\right).
\]
The generic norm and genus are
\[
N(z,z)
=
1-2\langle z,\bar z\rangle+|\langle z,z\rangle|^2,
\qquad
\gamma=n.
\]
The Bergman operator is
\[
\begin{aligned}
B(z,z)w
&=
N(z,z)\,w
\\
&\quad
+
\left(
2\langle\bar z,w\rangle(2\langle z,\bar z\rangle-1)
-
2\langle\bar z,\bar z\rangle\langle z,w\rangle
\right)z
\\
&\quad
+
\left(
2\langle z,w\rangle
-
2\langle z,z\rangle\langle\bar z,w\rangle
\right)\bar z .
\end{aligned}
\]
The canonical Hermitian trace form is
\[
(x|y)=2n\langle x,\bar y\rangle.
\]
The K\"ahler potential is
\[
\Phi(z)
=
-\frac12\log\left(
1-2\langle z,\bar z\rangle+|\langle z,z\rangle|^2
\right).
\]
The K\"ahler \(d^c\)-potential form is
\[
\alpha_z
=
\frac{i}{2N(z,z)}
\left(
\langle \xi,dz\rangle
-
\langle \bar\xi,d\bar z\rangle
\right),
\qquad
\xi=\langle\bar z,\bar z\rangle z-\bar z.
\]

This Hermitian positive Jordan triple system is of tube type. We have
\[
\mathcal J(z)=|\langle z,z\rangle|
\]
and
\[
    \theta_z
    =
    2\;
    \im\left(
    \frac{\langle z,dz\rangle}{\langle z,z\rangle}
    \right)
    ,
    \qquad
    z\in\Omega^\star
    .
\]

\subsection*{Type VI: exceptional domain of dimension \texorpdfstring{\(27\)}{}}

Let
\[
V=\mathcal H_3(\mathbb O_{\mathbb C}),
\]
the space of \(3\times 3\) matrices with entries in the complexified octonions
\(\mathbb O_{\mathbb C}\), Hermitian with respect to Cayley conjugation. Let
\(t(X,Y)\) denote the standard Hermitian product, let \(\det(X)\) denote the
determinant, and let \(X^\#\) denote the adjoint matrix of \(X\) in
\(\mathcal H_3(\mathbb O_{\mathbb C})\). The type-VI domain is
\[
\begin{aligned}
\Omega_{\mathrm{VI}}
=
\{Z\in\mathcal H_3(\mathbb O_{\mathbb C}):\;&
3-t(Z,Z)>0,
\\
&
3-2t(Z,Z)+t(Z^\#,Z^\#)>0,
\\
&
1-t(Z,Z)+t(Z^\#,Z^\#)-|\det(Z)|^2>0
\}.
\end{aligned}
\]
The triple product is
\[
\{X,Y,Z\}
=
t(X,Y)Z+t(Z,Y)X-(X\times Z)\times\bar Y,
\]
where \(\times\) denotes the Freudenthal product defined by polarization of
the adjoint operator:
\[
X\times Y=(X+Y)^\#-X^\#-Y^\#,
\qquad
X\times X=2X^\#.
\]
The generic norm and genus are
\[
N(Z,Z)
=
1-t(Z,Z)+t(Z^\#,Z^\#)-|\det(Z)|^2,
\qquad
\gamma=18.
\]
The Bergman operator is
\[
\begin{aligned}
B(Z,Z)W
&=
\left[1-t(Z,Z)\right]W
\\
&\quad
+
\left[
(t(Z,Z)-1)t(W,Z)
-
t(Z,Z^\#\times\bar W)
\right]Z
\\
&\quad
-
t(W,Z)\,(Z^\#\times\bar Z)
+
(Z\times W)\times\bar Z
+
Z^\#\times(\bar Z^\#\times W).
\end{aligned}
\]
The canonical Hermitian trace form is
\[
(X|Y)=18\,t(X,Y).
\]
The K\"ahler potential is
\[
\Phi(Z)
=
-\frac12\log\left(
1-t(Z,Z)+t(Z^\#,Z^\#)-|\det(Z)|^2
\right).
\]
The K\"ahler \(d^c\)-potential form is
\[
\alpha_Z
=
\frac{i}{4N(Z,Z)}
\left(
t(dZ,\Xi)
-
t(\Xi,dZ)
\right),
\]
where
\[
\Xi=\bar Z\times Z^\#-Z-\det(Z)\bar Z^\#.
\]

This Hermitian positive Jordan triple system is of tube type. We have
\[
\mathcal J(Z)=|\det(Z)|
\]
and
\[
    \theta_Z
    =
    \im \frac{t(dZ,\bar Z^\#)}{\det(Z)} 
    \qquad
    Z\in\Omega^\star
    .
\]

\subsection*{Type V: exceptional domain of dimension \texorpdfstring{\(16\)}{}}

Using the notation of the type-VI case, the type-V exceptional domain
corresponds to
\[
V=\mathcal M_{2,1}(\mathbb O_{\mathbb C}),
\]
which may be realized as the subspace of
\(\mathcal H_3(\mathbb O_{\mathbb C})\) consisting of matrices of the form
\[
\begin{pmatrix}
0 & a_3 & \tilde a_2 \\
\tilde a_3 & 0 & 0 \\
a_2 & 0 & 0
\end{pmatrix},
\]
where \(\tilde a\) denotes the Cayley conjugate of
\(a\in\mathbb O_{\mathbb C}\). The type-V domain is
\[
\begin{aligned}
\Omega_{\mathrm V}
=
\{Z\in\mathcal M_{2,1}(\mathbb O_{\mathbb C}):\;&
2-t(Z,Z)>0,
\\
&
1-t(Z,Z)+t(Z^\#,Z^\#)>0
\}.
\end{aligned}
\]
The triple product is
\[
\{X,Y,Z\}
=
t(X,Y)Z+t(Z,Y)X-(X\times Z)\times\bar Y.
\]
The generic norm and genus are
\[
N(Z,Z)
=
1-t(Z,Z)+t(Z^\#,Z^\#),
\qquad
\gamma=12.
\]
The Bergman operator is
\[
\begin{aligned}
B(Z,Z)W
&=
\left[1-t(Z,Z)\right]W
\\
&\quad
+
\left[
(t(Z,Z)-1)t(W,Z)
-
t(Z,Z^\#\times\bar W)
\right]Z
\\
&\quad
-
t(W,Z)\,(Z^\#\times\bar Z)
+
(Z\times W)\times\bar Z
+
Z^\#\times(\bar Z^\#\times W).
\end{aligned}
\]
The canonical Hermitian trace form is
\[
(X|Y)=12\,t(X,Y).
\]
The K\"ahler potential is
\[
\Phi(Z)
=
-\frac12\log\left(
1-t(Z,Z)+t(Z^\#,Z^\#)
\right).
\]
The K\"ahler \(d^c\)-potential form is
\[
\alpha_Z
=
\frac{i}{4N(Z,Z)}
\left(
t(dZ,\Xi)
-
t(\Xi,dZ)
\right),
\]
where
\[
\Xi=\bar Z\times Z^\#-Z.
\]
This Hermitian positive Jordan triple system is not of tube type.

\section{Auxiliary computations}
In this section we collect some computations used throughout the text to deal with the interacting term in the first-order part of $\mathcal{L}_{k_1,k_2,k_3}$.
\begin{lemma}\label{computation coth}
Let $a_1,\dots,a_r \in \mathbb{R}_{\ge 0}$ such that 
$a_i \neq a_j$ for $i \neq j$. Then
\[
\sum_{1 \le i \ne j \le r} 
\tanh a_i 
\bigl( \coth(a_i - a_j) + \coth(a_i + a_j) \bigr)
= r(r-1).
\]
\end{lemma}

\begin{proof}
Fix $i \ne j$. Using
\[
\tanh x = \frac{\sinh x}{\cosh x}, 
\qquad 
\coth x = \frac{\cosh x}{\sinh x},
\]
a direct computation with the usual hyperbolic identities gives
\[
\tanh a_i \coth(a_i - a_j)
= 1 + \frac{\sinh a_j}{\cosh a_i \sinh(a_i - a_j)},
\]
and
\[
\tanh a_i \coth(a_i + a_j)
= 1 - \frac{\sinh a_j}{\cosh a_i \sinh(a_i + a_j)}.
\]
Adding these identities yields
\[
\tanh a_i \bigl( \coth(a_i - a_j) + \coth(a_i + a_j) \bigr)
=
2
+
\frac{\sinh a_j}{\cosh a_i}
\left(
\frac{1}{\sinh(a_i - a_j)}
-
\frac{1}{\sinh(a_i + a_j)}
\right).
\]
Using
\[
\sinh(a_i+a_j) - \sinh(a_i-a_j)
= 2 \cosh a_i \sinh a_j
\]
and
\[
\sinh(a_i-a_j)\sinh(a_i+a_j)
=
\sinh^2 a_i - \sinh^2 a_j,
\]
we obtain
\[
\tanh a_i 
\bigl( \coth(a_i - a_j) + \coth(a_i + a_j) \bigr)
=
\frac{2 \sinh^2 a_i}{\sinh^2 a_i - \sinh^2 a_j}.
\]

Therefore,
\[
\sum_{i \ne j}
\tanh a_i 
\bigl( \coth(a_i - a_j) + \coth(a_i + a_j) \bigr)
=
\sum_{i \ne j}
\frac{2 \sinh^2 a_i}{\sinh^2 a_i - \sinh^2 a_j}.
\]
For each unordered pair $\{i,j\}$,
\[
\frac{2 \sinh^2 a_i}{\sinh^2 a_i - \sinh^2 a_j}
+
\frac{2 \sinh^2 a_j}{\sinh^2 a_j - \sinh^2 a_i}
=
2.
\]
Since there are $\binom{r}{2}$ such pairs, the sum equals
\[
2 \binom{r}{2} = r(r-1).
\]
\end{proof}

\begin{lemma}\label{computation coth bis}
Let $a_1,\dots,a_r \in \mathbb{R}_{>0}$ such that 
$a_i \neq a_j$ for $i \neq j$. Then
\[
    \sum_{i \ne j} \coth(2a_i)\bigl(\coth(a_i-a_j)+\coth(a_i+a_j)\bigr)
    = r(r-1).
\]
\end{lemma}

\begin{proof}
Let $S$ denote the left hand side of the equation.
Using the identity
\[
\coth u + \coth v = \frac{\sinh(u+v)}{\sinh u \, \sinh v},
\]
with \(u = a_i - a_j\), \(v = a_i + a_j\), we get
\[
\coth(a_i-a_j) + \coth(a_i+a_j)
= \frac{\sinh(2a_i)}{\sinh(a_i-a_j)\sinh(a_i+a_j)}.
\]

Therefore,
\[
S = \sum_{i \ne j} 
\coth(2a_i)\frac{\sinh(2a_i)}{\sinh(a_i-a_j)\sinh(a_i+a_j)}
= \sum_{i \ne j} 
\frac{\cosh(2a_i)}{\sinh(a_i-a_j)\sinh(a_i+a_j)}.
\]

Now group terms for pairs \((i,j)\) and \((j,i)\):
\[
\frac{\cosh(2a_i)}{\sinh(a_i-a_j)\sinh(a_i+a_j)}
+
\frac{\cosh(2a_j)}{\sinh(a_j-a_i)\sinh(a_j+a_i)}.
\]

Since
\[
\sinh(a_j-a_i) = -\sinh(a_i-a_j), \quad 
\sinh(a_j+a_i) = \sinh(a_i+a_j),
\]
this becomes
\[
\frac{\cosh(2a_i)-\cosh(2a_j)}{\sinh(a_i-a_j)\sinh(a_i+a_j)}.
\]

Using
\[
\cosh(2a_i) - \cosh(2a_j)
= 2\sinh(a_i+a_j)\sinh(a_i-a_j),
\]
each pair contributes \(2\). Hence
\[
S = 2 \binom{r}{2} = r(r-1).
\]
\end{proof}

\bibliographystyle{plain}
\bibliography{biblio}

\vspace{5pt}
\noindent
\begin{minipage}{\textwidth}
    \small
    \textbf{Fabrice Baudoin:} \\
    Department of Mathematics, Aarhus University \\
    Email: fbaudoin@math.au.dk
\end{minipage}

\vspace{10pt} 
\noindent
\begin{minipage}{\textwidth}
    \small
    \textbf{Alexandre Reber:} \\
    Department of Mathematics, Aarhus University \\
    Email: alexandre.reber@math.au.dk
\end{minipage}

\end{document}